\documentclass[a4paper]{article}
\usepackage{amsmath,amsthm,amsfonts,amssymb}
\usepackage{graphicx}
\usepackage{color}
\usepackage{a4wide}
\usepackage{cases}
\usepackage[english]{babel}
\usepackage[numbers]{natbib}

\usepackage{enumitem}

\usepackage{caption}
\usepackage{subcaption}

\title{Polyhedral geometry, supercranks, and combinatorial witnesses of congruences for partitions into three parts}
\author{Felix Breuer\thanks{Felix Breuer and Brandt Kronholm were supported by the Austrian Science Fund (FWF) special research group Algorithmic and Enumerative Combinatorics SFB F50, project number F5006-N15.
}\\Research Institute for Symbolic Computation (RISC)\\ Johannes Kepler University, A-4040 Linz, Austria\\ \texttt{breuer@risc.jku.at}
\and
Dennis Eichhorn\\Department of Mathematics \\ University of California, Irvine \\ Irvine, CA 92697-3875\\ \texttt{deichhor@math.uci.edu}
\and
Brandt Kronholm\footnotemark[1]\\Research Institute for Symbolic Computation (RISC)\\ Johannes Kepler University, A-4040 Linz, Austria\\ \texttt{kronholm@risc.jku.at}
}

  \renewcommand{\leq}{\leqslant}
  \renewcommand{\geq}{\geqslant}
  \newcommand{\PPP}[0]{\mathcal{P}} 
   \newcommand{\CCC}[0]{\mathcal{C}} 
     \newcommand{\FFF}[0]{\mathcal{F}} 
     \newcommand{\HHH}[0]{\mathcal{H}} 
        \newcommand{\RRR}[0]{\mathcal{R}} 
   
  \newcommand{\Z}[0]{\mathbb{Z}}		
  \newcommand{\R}[0]{\mathbb{R}}		
  \newcommand{\mset}[2]{\left\{ #1 \; \middle|\; #2 \right\}}
    
  \newcommand{\mmod}{\operatorname{mod}} 
  
  \newcommand{\fract}[1]{\left\{ #1 \right\}} 
  \newcommand{\floor}[1]{\left\lfloor #1 \right\rfloor} 

  \newcommand{\mmat}[1]{\begin{pmatrix} #1\end{pmatrix}}
  \newcommand{\msmat}[1]{\left(\begin{smallmatrix} #1\end{smallmatrix}\right)}
  \newcommand{\rar}{\rightarrow}
  
\makeatletter
\newtheorem*{rep@theorem}{\rep@title}
\newcommand{\newreptheorem}[2]{%
\newenvironment{rep#1}[1]{%
 \def\rep@title{#2 \ref{##1}}%
 \begin{rep@theorem}}%
 {\end{rep@theorem}}}
\makeatother

    \newtheorem{theorem}{Theorem}[section]
    \newreptheorem{theorem}{Theorem}
  \newtheorem*{theorem*}{Theorem}
    \newtheorem{lemma}[theorem]{Lemma}
    
    \newtheorem{corollary}[theorem]{Corollary}

    \newtheorem{proposition}[theorem]{Proposition}
    
  \theoremstyle{definition}
    \newtheorem{definition}[theorem]{Definition}

\definecolor{antired}{rgb}{0,0.6,0.4}

  \newcommand{\rprop}[1]{Proposition~\ref{#1}}
  
  \newcommand{\rthm}[1]{Theorem~\ref{#1}}

  \newcommand{\rsec}[1]{Section~\ref{#1}}

\makeatletter
\def\lcm#1{\allowbreak\mkern1mu{\operator@font lcm}(#1)}
\makeatother

\makeatletter
\def\imod#1{\allowbreak\mkern9mu({\operator@font mod}\,\,#1)}
\makeatother

\begin{document}                \large
\maketitle
\begin{abstract}
In this paper, we use a branch of polyhedral geometry, Ehrhart theory,
to expand our combinatorial understanding of congruences for partition
functions.
Ehrhart theory allows us to give a new decomposition of partitions,
which in turn allows us to define statistics called {\it supercranks}
that combinatorially witness every instance of divisibility of
$p(n,3)$ by any prime $m \equiv -1 \pmod 6$, where $p(n,3)$ is
the number of partitions of $n$ into
three parts.
A rearrangement of lattice points allows us to
demonstrate with explicit bijections how to divide these sets of partitions into $m$ equinumerous classes.
The behavior for primes $m' \equiv 1 \pmod 6$ is also discussed.

\end{abstract}

\section{Introduction}
\label{introduction}
Ramanujan  \cite{Ram} observed and proved the following congruences for the ordinary partition function:
\[ p(5n+4)\equiv0\pmod5\]
\[ p(7n+5)\equiv0\pmod7 \]
\[ p(11n+6)\equiv0\pmod{11}. \]

In 1944, Freeman Dyson \cite{Dyson} called for {\it direct} proofs of Ramanujan's congruences
that would give concrete demonstrations of how the sets of partitions of $5n+4$, $7n+5$, and $11n+6$ could be split into
$5, 7,$ and $11$ equinumerous classes, respectively.

\begin{quote} ...it is unsatisfactory to receive no concrete idea of how the division is to be made.  We require a proof which will not appeal to generating functions, but will demonstrate by cross-examination of the partitions themselves...\cite{Dyson}
\end{quote}
He conjectured that a very simple statistic on partitions,
the largest part minus the smallest part,
performs this division when considered modulo 5 and 7.
He named this statistic the ``rank" of a partition,
and he further hypothesized the existence of a different statistic,
called the ``crank,'' that would witness Ramanujan's congruence modulo 11 in the same way.
In \cite{AS}, Atkin and Swinnerton-Dyer proved Dyson's conjecture about the rank,
and in 1988, Andrews and Garvan \cite{Andrews2} found a crank that not only witnessed Ramanujan's congruence modulo 11, but also
witnessed Ramanujan's congruences modulo 5 and 7
with a new division into 5 and 7 classes, respectively.
However, in both cases, the proofs were analytic, and they did not
employ a cross-examination of the partitions themselves
as Dyson had hoped.

In \cite{GKS}, Garvan, Kim, and Stanton finally gave a combinatorial proof of Ramanujan's congruences by finding explicit bijections among
equinumerous classes.
These bijections were realized by 5-, 7-, and 11-cycles, respectively,
that exhaust the corresponding sets of partitions.
Their cycles led to new crank statistics that were different from the rank
and the Andrews-Garvan crank, but that witnessed the congruence modulo 11 as Dyson had requested,
and also witnessed the congruences modulo 5 and 7
with another new division into 5 and 7 classes, respectively.
Remarkably,
more than 70 years after Dyson's original request,
there is still no known bijective proof that the rank witnesses Ramanujan's congruences modulo 5 and 7.

In this paper, we explore the possibilities for giving bijective proofs of partition
congruences by considering $p(n,3)$, the number of partitions of $n$ into three parts.
We define a new crank-like statistic on partitions that has a remarkable property.
Unlike
the rank and the few known
cranks for $p(n)$ which witness congruences along certain arithmetic progressions,
we discover cranks for $p(n,3)$ that witnesses {\it each and every} instance of divisibility modulo a given prime.
We call cranks with this remarkable property {\it supercranks},
and they were first treated by the second and third authors in \cite{EichhornKronholm-forthcoming}.
The new techniques we introduce here allow us to give the first infinite family of
supercranks, and more generally establish an entirely new framework for treating
certain types of partition congruences and crank statistics.
Although we currently restrict our attention to $p(n,3)$, in
\cite{BreuerEichhornKronholm-forthcoming} we address a much more general class of
partition functions.

Our new method for discovering bijective proofs comes from an appeal to polyhedral geometry, specifically, Ehrhart Theory \cite{Beck,Breuer,Ehrhartpolynomial,Ehrhart1,Ehrhart2}.
Partitions of an integer $n$ into three parts can be viewed as integer vectors lying inside a triangle in 3-dimensional space.
We call this triangle the partition triangle, and
the natural inequalities that determine partitions with three parts define a polyhedral cone that we call the partition cone.
Following Ehrhart, the partition cone can be tiled by integer translates of a certain fundamental parallelepiped.
This tiling of the partition cone then induces a tiling of the partition triangle with slices of the fundamental parallelepiped (Figures~\ref{fig:cone-tiling}, \ref{fig:master-tile-set} and \ref{fig:tiling-examples}).
This construction allows us to decompose each such partition $\lambda$ into a partition $\mu$ in the fundamental parallelepiped and a non-negative integer vector $\tau$ that determines which translate of the fundamental parallelepiped $\lambda$ lies in.
We call $\mu$ the \emph{box remainder} and $\tau$ the \emph{box quotient}, which together give the \emph{box decomposition} of $\lambda$ (Figure~\ref{fig:decomposition} and Definition~\ref{def:decomposition}).
The box decomposition has a purely combinatorial description that we study in detail in \cite{BreuerEichhornKronholm-forthcoming}.
In this paper, we maintain a geometric vantage point whereby the box decomposition allows us to view the set of all partitions of $n$ into three parts as a union of six copies of triangular arrays $T_k$ of lattice points.
In particular, as Ehrhart already pointed out, this provides a geometric interpretation of the coefficients of the quasipolynomial formula for $p(n,3)$ in a binomial basis.

With this geometric insight in hand, when $p(n,3)$ exhibits divisibility,
it is often possible to reassemble the triangles $T_k$ into a rectangle of lattice points wherein the number of lattice points along the width or height of this rectangle is divisible by the modulus $m$ we are interested in (Figure~\ref{fig:unlabeled-2m-2}).
Cycling partitions along the rows of this rectangle provides a combinatorial witness for divisibility.
Since this construction works for any arrangement of triangles $T_k$ into a suitable rectangle, this method provides a whole family of combinatorial witnesses for congruences of $p(n,3)$ (Theorems~\ref{InformalTheorem} and \ref{thm:unlabeled-technical}).
Cranks for every such witness are given by a composition of the piecewise linear functions that perform the
original decomposition into triangles $T_k$ and the subsequent rearrangement into a rectangle.

In general, essentially the only way to write down a formula for these cranks is by this direct appeal
to the mechanics of this rearrangement.
However, for every prime $m \equiv -1 \pmod 6$,
there is one special scheme for rearranging these triangles which causes all of the resulting cranks
to simplify into a single simple formula: the largest part minus the smallest part modulo $m$.
This simple statistic, denoted $c_{LS}$, witnesses each and every divisibility of $p(n,3)$ by any such $m$.
In other words, the polyhedral geometry approach provides us with a deep structural insight into the set of partitions into three parts, allowing us to construct a supercrank (Theorem~\ref{main theorem}).

The rest of the paper is arranged as follows.
In Section 2, we give background on formulas for $p(n,3)$
and fully characterize when $p(n,3)$ is divisible by any prime $m \equiv -1 \pmod 6$.
In Section 3, we define terms surrounding combinatorial proofs of partition congruences,
and we state our main theorem.
In Section 4, we give our key decomposition of partitions into three parts via Ehrhart theory.
In Section 5, we use this decomposition to demonstrate a general method for constructing Ehrhart
cranks that witness congruences for $p(n,3)$.
In Section 6, we prove our main theorem based on the mathematics developed in the previous sections.
Section 7 provides an alternate proof of the main theorem which, in contrast to Section 6, is confined to only the partition triangle.
In Section 8, we use our techniques to consider the divisibility of $p(n,3)$ by primes $m' \equiv 1 \pmod 6$.
In Section 9, we offer some concluding remarks and indicate some directions for future study.

\section{The Arithmetic and Congruence Properties of $p(n,3)$}
\label{sec:arithmetic}
\subsection{Historical Background for $p(n,3)$.}

Historically, there have been several ways to compute the values of $p(n,3)$.
One method is the expansion of the generating function for $p(n,3)$ as rational function.

\begin{equation}\label{gf}
\sum_{n=0}^{\infty}p(n,3)q^n = \frac{q^3}{(1-q)(1-q^2)(1-q^3)}=q^3\times\sum_{j=0}^{\infty}q^j\times\sum_{j=0}^{\infty}q^{2j}\times\sum_{j=0}^{\infty}q^{3j}.
\end{equation}

However, closed form formulas are far more convenient and attractive.
In the middle of the nineteenth century, DeMorgan \cite{DeMorgan,Dickson} proved that $p(n,3)$ was the nearest integer to
\begin{equation}\label{n^2/12}
\frac{n^2}{12},
\end{equation}
and Warburton \cite{Dickson,Warburton} established
\begin{equation}\label{p(6k+r,3)}
    p(6k+r,3)=\left\{
      \begin{array}{lr}
      3k^2 & ~~~~r=0~\\
      3k^2 + k & ~~~~r=1~\\
      3k^2 + 2k & ~~~~r=2~\\
      3k^2+3k+1 & ~~~~r=3~\\
      3k^2+4k+1 & ~~~~r=4~\\
      3k^2+5k+2 & ~~~~r=5.
      \end{array}
 \right.
 \end{equation}

Following the work of Herschel \cite{Herschel},
by the turn of the 20th century even more methods for computing $p(n,3)$ had been developed by
Cayley \cite{Cayley1}, Sylvester \cite{Sylvester}, Glaisher \cite{Glaisher1}, and others \cite{Dickson,GuptaT}.
For example:
\begin{equation}\label{circulators}
p(n,3)=\frac{n^2-\frac{7}{6}}{12}-\frac{(-1)^n}{8}+\frac{e^{\frac{2i\pi n}{3}}+e^{\frac{4i\pi n}{3}}}{9}.
\end{equation}
Their efforts were not solely focused on $p(n,3)$, but on $p(n,d)$ in general as it applied to the Theory of Invariants.
Each of (\ref{n^2/12}), (\ref{p(6k+r,3)}), and (\ref{circulators}) are considered quasipolynomials for $p(n,3)$.

Unfortunately, the above expressions for $p(n,3)$ have some shortcomings.
The methods used to obtain them either do not easily generalize, or they require mathematics with a significant amount of depth.
Moreover
the expressions in (\ref{n^2/12}), (\ref{p(6k+r,3)}), and (\ref{circulators}) elicit
very little information about the partitions themselves.
In the next section, we consider an alternative that eradicates all of these shortcomings.

\subsection{Ehrhart's method for computing quasipolynomials.}

Half a century ago, the French geometer Eug\`{e}ne Ehrhart devised a very elegant method for computing formulas for functions such as $p(n,3)$, $p(n,d)$,
 and many others \cite{Beck,Ehrhartpolynomial,Ehrhart1,Ehrhart2}.
While Ehrhart was interested in counting integer points in dilates of polytopes, or, in other words,
the number of solutions of a linear Diophantine system as a function of the system's right-hand side,
his method boils down on the arithmetic level to straightforward manipulation of the relevant generating function.
The key insight is that it is useful to bring the denominator into the form $(1-q^j)^d$. For $p(n,3)$, we compute:

{\normalsize
\begin{equation}\label{GFa}
\sum_{n=0}^{\infty}p(n,3)q^n = \frac{q^3}{(1-q)(1-q^2)(1-q^3)} = \frac{q^3(1+q+q^2+q^3+q^4+q^5)(1+q^2+q^4)(1+q^3)}{(1-q^6)(1-q^6)(1-q^6)}
\end{equation}
\begin{equation}\label{GFb}
=\frac{q^3+q^4+2q^5+3q^6+4q^7+5q^8+4q^9+5q^{10}+4q^{11}+3q^{12}+2q^{13}+q^{14}+q^{15}}{(1-q^6)^3}
\end{equation}
\begin{equation}\label{GFc}
=(q^3+q^4+2q^5+3q^6+4q^7+5q^8+4q^9+5q^{10}+4q^{11}+3q^{12}+2q^{13}+q^{14}+q^{15})\times\sum_{k=0}^{\infty}{k+2 \choose 2}q^{6k}.
\end{equation}}

Hence,

\begin{eqnarray}
p(6k,3) & = 0{k+2 \choose 2} + 3{k+1 \choose 2} +  3{k \choose 2} \label{p(6k,3)}\\
p(6k+1,3) & = 0{k+2 \choose 2} + 4{k+1 \choose 2} +  2{k \choose 2}\label{p(6k+1,3)}\\
p(6k+2,3) & = 0{k+2 \choose 2} + 5{k+1 \choose 2} +  1{k \choose 2} \label{p(6k+2,3)}\\
p(6k+3,3) & = 1{k+2 \choose 2} + 4{k+1 \choose 2} + 1{k \choose 2}  \label{p(6k+3,3)}\\
p(6k+4,3) & = 1{k+2 \choose 2} + 5{k+1 \choose 2} + 0{k \choose 2}  \label{p(6k+4,3)}\\
p(6k+5,3) & = 2{k+2 \choose 2} + 4{k+1 \choose 2} + 0{k \choose 2}. \label{p(6k+5,3)}
\end{eqnarray}

In this paper, our focus is on the {\it structure underlying} lines (\ref{p(6k,3)}) through (\ref{p(6k+5,3)}).
We will show that these expressions contain fundamental information about arithmetic, geometric, and combinatorial properties of the partitions themselves.
Notice that (\ref{p(6k,3)}) through (\ref{p(6k+5,3)}) are given in the binomial basis which, when simplified to the monomial basis, are equivalent to the expressions in (\ref{p(6k+r,3)}).

In particular, the formulas (\ref{p(6k,3)}) through (\ref{p(6k+5,3)}) as well as (\ref{n^2/12}), (\ref{p(6k+r,3)}),
and (\ref{circulators}) show that $p(n,3)$ is a quasipolynomial, as are all $p(n,d)$ for fixed $d$.
A \emph{quasipolynomial} $\Pi(k)$ is a polynomial in $k$ whose coefficients are periodic functions of $k$.
Equivalently, a quasipolynomial is a function $\Pi$ such that there exist polynomials $\Pi_0,\ldots,\Pi_{l-1}$ with the property
$\Pi(lk+r) = \Pi_{r}(k)$ for all $k\geq 0$ and $r=0,\ldots,l-1$.
The integer $l$ is called a \emph{period} of $\Pi$. The minimal period of the partition function $p(n,3)$ is $6$ as we have seen above.
In general the minimal period of $p(n,d)$ is $\lcm{d}:=\lcm{1,\ldots,d}$,
which is implicit in the work of
Herschel \cite{Herschel}, Cayley \cite{Cayley1}, Sylvester \cite{Sylvester}, Glaisher \cite{Glaisher1}, and others.

\subsection{Establishing Infinite Families of Divisibility for $p(n,3)\pmod m$}

With a quasipolynomial formula in hand,
it is straightforward to determine various divisibility properties of $p(n,3)$.
For any prime $m \equiv -1 \pmod 6$, we can determine
when $m | p(n,3)$ completely.

\begin{proposition}\label{6j-1prop}
Let $m  = 6j-1$ be prime.
Then
$$
p(n,3) \equiv 0 \pmod m  \text{\quad if and only if \quad} n \equiv \pm (0,1,2,2m-2,2m+1) \pmod {6m} .
$$
\end{proposition}

\noindent
We will often write these congruences as,
for $k\geq 0$,
\begin{equation}
p(6mk\pm 0,1,2,2m-2,2m+1;3)\equiv 0\pmod{m}.
\end{equation}

\begin{proof}

Write $n  = 6k +r$ for $-2 \leq r \leq 3$ and appeal to (\ref{p(6k+r,3)}).

\noindent
{\bf Case 1:}  $r=0$.

Then $m|3k^2 \Longleftrightarrow m|k$; i.e., $6k = n \equiv 0 \pmod {6m}$.

\noindent
{\bf Case 2:}  $r=\pm 1$ or $\pm 2$.

Then $m|(3k^2 + rk) \Longleftrightarrow m|k$ (i.e., $n \equiv r \pmod {6m}$)
    or $m|(3k + r)$.
Now $m|(3k + r) \Longleftrightarrow k \equiv -r/3 \pmod m$, and since $m = 6j-1$, the inverse of $3$
is $2j$, and so this is the same as $k \equiv -2jr \pmod m$.
Multiplying both sides of this congruence by $6$, adding $r$, and observing that
$12j = 2m+2$, we have that this is equivalent to
$$
6k+r = n \equiv -12rj+r = -2mr-r \pmod {6m} .
$$

Thus, in Case 2,
we have

\noindent
$p(n,3) \equiv 0 \pmod m \Longleftrightarrow n \equiv r \pmod {6m} \text{\ or\ } n \equiv -2mr-r \pmod {6m}$;

\noindent
i.e., exactly when $n \equiv \pm (1,2,2m+1,4m+2) \equiv \pm (1,2,2m-2,2m+1) \pmod {6m}$.

\noindent
{\bf Case 3:}  $r=3$.

Observe that $3k^2+3k+1 = 3(k+1/2)^2 + 1/4$.  Also observe that, modulo $m = 6j-1$, the inverse
of $2$ is $3j$, and so the inverse of $4$ is $9j^2$.
Thus $m | p(6k+3, 3)$ if and only if
$$
3k^2+3k+1 = 3(k+1/2)^2 + 1/4 \equiv 3[(k+3j)^2 + 3j^2] \equiv 0 \pmod m .
$$
Dividing by 3, this holds if and only if
$(k+3j)^2  \equiv -3j^2 \pmod m$,
which can only have solutions if $-3$ is a square modulo $m$.
Since
$$
\left ( \frac{-1}{m} \right ) = (-1)^{(m-1)/2} \qquad \text{and} \qquad
\left ( \frac{3}{m} \right ) = (-1)^{\lfloor (m+1)/6 \rfloor} ,
$$
we see that no matter what the parity of $j$ is, $-3$ is not a square modulo $m$.
\end{proof}

\section{Crank, Supercrank, and Statement of Main Theorem}

We saw in the last section that modulo any prime $m \equiv -1 \pmod 6$,
the quasipolynomial formula for $p(n,3)$ allows us to see exactly when $m | p(n,3)$.
A natural question to ask is,
``is there a combinatorial way we could have predicted this without appealing to the formula?"
Also, ``are there crank statistics that witness any of this divisibility?"
The remainder of this paper is devoted to demonstrating affirmative answers to both of
these questions.
In this section, we make rigorous the notions of ``combinatorial witness,'' ``rank,"
and ``crank," we state our main theorem, and we outline the proof.

\begin{definition}
Let $S$ be a finite set of size $N=|S|$, and let $m$ be a positive integer. A \emph{combinatorial witness} for $m|N$ is an explicit partition
$$ S = S_0 \sqcup S_1 \sqcup \ldots \sqcup S_{m-1} $$
of $S$ into disjoint sets $S_i$ together with explicit bijections
$$ \phi_{i,j}:S_i \rightarrow S_j $$
which witness that any two of the $S_i$ are of the same size.
\end{definition}

Dyson's original request for a crank statistic has now been satisfied in a few different
ways.  Here, we will refer to any statistic that forms a combinatorial witness for
divisibility as a crank.

\begin{definition}
A \emph{crank} is a function $c:S\rightarrow \{0,\ldots,m-1\}$ defining the classes $S_i = c^{-1}(i)$ and \emph{cycles} are given by a permutation $\pi:S\rightarrow S$ such that $\pi|_{S_i}$ is a bijection between $S_i$ and $S_{i+1}$, where the index $i+1$ is understood modulo $m$.
\end{definition}

In this article, we introduce a new class of cranks, which we call \emph{Ehrhart cranks}, in Sections~\ref{sec:ehrhart} and \ref{sec:Ehrhart-cranks} below.
Ehrhart cranks make excellent combinatorial witnesses as they reveal a tremendous amount of structure in sets of partitions with a restricted number of parts.

One remarkable crank is the following. Let $P(n,3)$ denote the set of partitions $\lambda=\lambda_1+\lambda_2+\lambda_3$ with three parts $\lambda_1\geq \lambda_2 \geq \lambda_3>0$. Let a modulus $m$ be fixed. For every $\lambda\in P(n,3)$ define
\begin{eqnarray}
\label{eqn:cLS}
  c_{LS}(\lambda) := \lambda_1 - \lambda_3 \mod m,
\end{eqnarray}
to be the largest part minus the smallest part modulo $m$. As we will show in Section~\ref{labeled origami}, $c_{LS}$  defines a crank for \emph{all} the congruences given in Proposition~\ref{6j-1prop}. This is a particularly strong property for which we coin a new name: {\it Supercrank}.

\begin{definition}
Let $m>1$ be a fixed integer. Let $S(n)$ be a finite set for every nonnegative integer $n$.
A function $c$ defined on $\bigcup_{n\in  \Z_{\geq 0}} S(n)$ is a \emph{supercrank} if for \emph{all} $n\in  \Z_{\geq 0}$ such that
$m|\#S(n)$, the function $c$ is a crank witnessing this divisibility.
\end{definition}

A main result of this paper is that $c_{LS}$ has this property
for the set of partitions into three parts for every prime $m \equiv -1 \pmod 6$.

\begin{theorem}\label{main theorem}
Let $m=6j-1$ be prime. Then $c_{LS}$, largest part minus smallest part modulo $m$, is a supercrank witnessing $m|p(n,3)$ for each and every $n\in  \Z_{\geq 0}$ for which this divisibility holds, as characterized in Proposition~\ref{6j-1prop}.
\end{theorem}

We now give a rough outline of how \rthm{main theorem} will be established,
and we begin by describing a way in which one might geometrically demonstrate
divisibility of $p(n,3)$ by $m$.

By treating a partition $n=\lambda_1 + \lambda_2 + \lambda_3$ as an integer vector
$\lambda=(\lambda_1,\lambda_2,\lambda_3)\in\Z^3$,
we can see that the set of partitions of $n$ into three parts is
\[
P(n,3) = \mset{(\lambda_1,\lambda_2,\lambda_3) \in\Z^3}{\lambda_1 + \lambda_2 + \lambda_3 = n \ \text{and} \  \lambda_1 \geq \lambda_2 \geq \lambda_3 > 0}.
\]
As can been seen in the example shown in Figure~\ref{fig:partition-triangle},
these lattice points fit into an obvious triangular region $\PPP(n,3)$, which is determined by the above equation and inequalities
applied in $\R^3$.
If we find a way to rearrange these points evenly into a rectangle such that
the number of lattice points along the width or height of the rectangle is divisible by the modulus $m$ we are interested
in, then we have a geometric demonstration that $p(n,3)$ is divisible by $m$
(see Figure \ref{fig:unlabeled-2m-2} for an illustration).
Of course, we would like to do this not for one particular $n$,
but, if possible, for every single $n$ such that $m | p(n,3)$.
By studying these lattice points in a way outlined by Ehrhart \cite{Beck,Ehrhartpolynomial,Ehrhart1,Ehrhart2},
a great deal of structure is revealed in Section~\ref{sec:ehrhart}.
In particular, for every $n$, this collection of lattice points dissects nicely into six
neatly arranged triangular collections of points (see Figure~\ref{fig:decomposition}).
As it turns out,
for every prime $m \equiv -1 \pmod 6$,
we are able to find a uniform method for arranging these triangles into rectangles with
side lengths divisible by $m$ for every $n$ such that $m | p(n,3)$,
and thus we have a geometric proof of the divisibility.

In fact, the rectangles proving divisibility offer an obvious way to divide the
partitions into cycles of length $m$.
We can define a ``crank" statistic on the lattice points (and equivalently, on the partitions) as simply
``the distance from the appropriate edge of the rectangle,"
and then our partitions divide into $m$ equal classes
according to their ``crank" modulo $m$.
As it turns out, whenever we have one arrangement of our triangles into a useful rectangle,
we actually have many such arrangements.
For each, we get a different crank statistic that witnesses the divisibility. We call cranks of this type ``Ehrhart cranks", which are defined in \rsec{sec:Ehrhart-cranks}.

In \rsec{labeled origami}, we find that among all of the possible arrangements of triangles into rectangles,
there is one that is by far the most well poised.
There is one way in particular of arranging the triangles
such that a constant multiple of the distance modulo $m$
from one
edge of the rectangle is
identically the largest part minus the smallest part of the partition.
In other words,
we have \rthm{main theorem}, a statement as simple as Dyson's original conjecture.
What is quite striking is that, whereas Dyson's rank witnesses the first two Ramanujan congruences,
this new supercrank
witnesses \emph{every} congruence of the form
$p(n,3) \equiv 0 \pmod m$
for \emph{every} prime $m \equiv -1 \pmod6 $.


\section{The Box Decomposition of Restricted Partitions}
\label{sec:ehrhart}

In this section, we introduce the box decomposition of a restricted partition $\lambda$ into a box remainder $\mu$,
which is a partition in the fundamental parallelepiped $\FFF_3$ defined below,
and a box quotient $\tau$, which is a non-negative integer vector.
This decomposition is motivated by polyhedral geometry, and is the result of applying a classic construction in Ehrhart theory \cite{Beck,Ehrhartpolynomial,Ehrhart1,Ehrhart2} in a partition theoretic context.
While this decomposition can be defined purely in combinatorial terms \cite{BreuerEichhornKronholm-forthcoming}, the geometric point of view will provide the key intuition for the rest of this paper, and so we introduce the relevant background from Ehrhart theory in this section.
In particular, we take great care to visualize the construction in order to build geometric intution.
To be clear, our use here of the word {\it box} is not motivated by the geometry at hand but by the Ferrers Diagram of $\lambda$.
Since partitions into three parts are the topic of this paper, we will restrict our attention to the three-dimensional case.
Note, however, that the constructions below generalize in a straightforward manner to partitions with any fixed number of parts; this is treated in detail in \cite{BreuerEichhornKronholm-forthcoming}.
For general introductions to Ehrhart theory and polyhedral geometry, we recommend the textbooks \cite{Beck} and \cite{Ziegler}.

\begin{figure*}[t]
  \centering
   \includegraphics[angle=270,width=7cm]{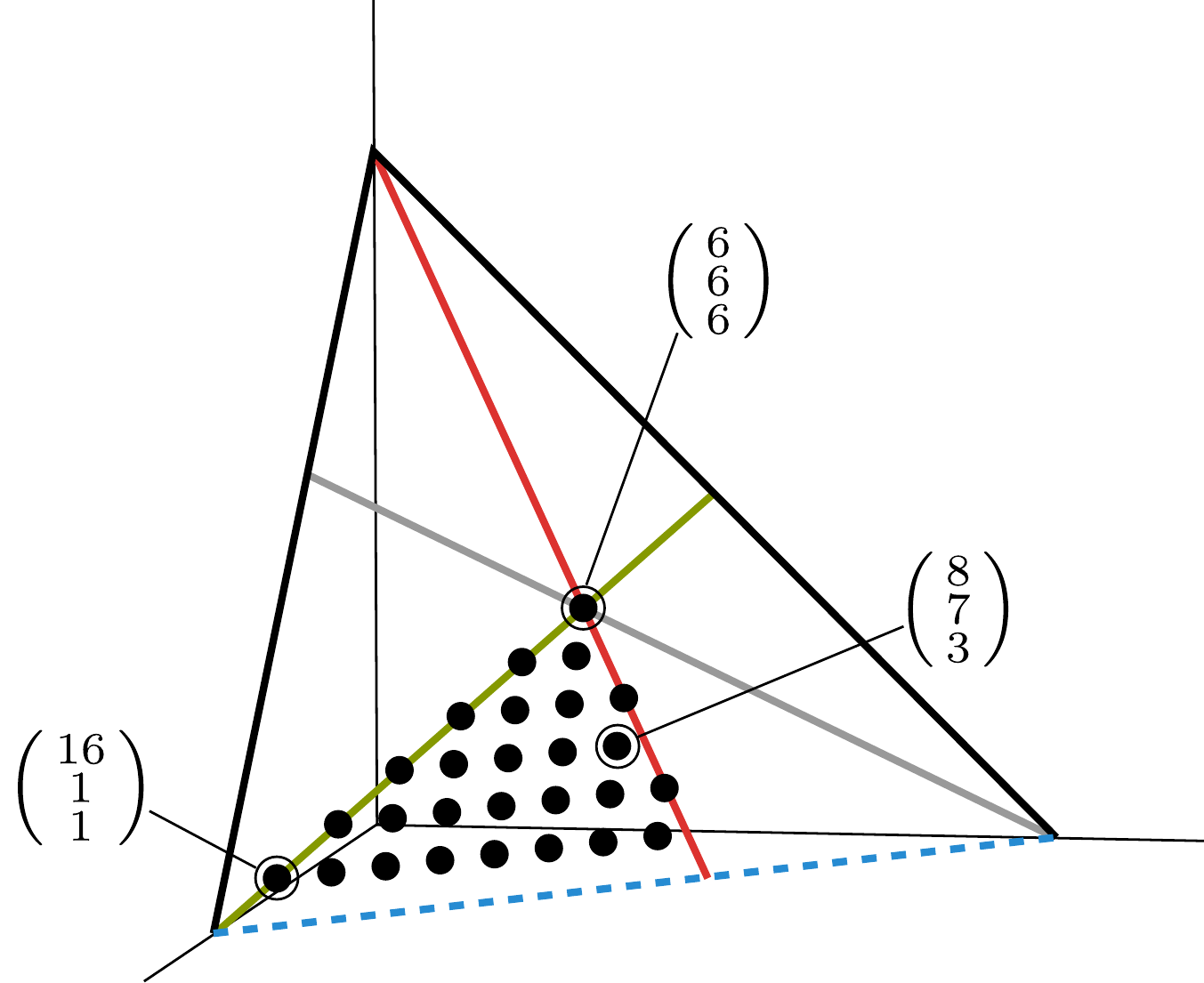}
  \caption{ \label{fig:partition-triangle} The partition triangle $\PPP(18,3)$ and the set $P(18,3) = \Z^3 \cap \PPP(18,3)$ of partitions it contains. The outer triangle is the intersection of the non-negative octant with the constraint $x_1+x_2+x_3=18$. Adding the inequality constraints $x_1\geq x_2$ (shown in red), $x_2 \geq x_3$ (green) and $x_3 > 0$ (dashed blue) yields the half-open partition triangle $\PPP(18,3)$. The lattice points $\lambda\in P(18,3)$ are shown as black dots.}
\end{figure*}

We now set the stage for our definition of the box decomposition in Definition~\ref{def:decomposition} below.
As illutrated in Figure~\ref{fig:partition-triangle}, the starting point for the geometric approach is to view
a partition $n=\lambda_1 + \lambda_2 + \lambda_3$ as an integer vector
$\lambda=(\lambda_1,\lambda_2,\lambda_3)\in\Z^3$.
Experienced partition theorists may not be used to this convention, so just to be
absolutely clear, throughout the rest of the paper, we will literally use the word
``partition" to mean such a vector in $\Z^3$.
The \emph{height} $|\lambda|=\lambda_1+\lambda_2+\lambda_3$ is the sum of coordinates of $\lambda$; i.e., the number being partitioned. The set of all partitions into three parts is then the set $C_3=\Z^3\cap \CCC_3$ of integer vectors or \emph{lattice points} in the \emph{partition cone}

\begin{equation}\label{p(n,d)cone}
  \CCC_3  =  \mset{x\in\R^{3}}{ x_1 \geq x_2 \geq x_3 > 0}.
\end{equation}
The set of partitions of a fixed $n$ into three parts is then the set $P(n,3)=\Z^3 \cap \PPP(n,3)$ of lattice points in the \emph{partition polytope} or \emph{partition triangle}
\begin{equation}\label{p(n,d)triangle}
  \PPP(n,3) = \CCC_3 \cap \mset{x\in\R^{3}}{ |x| = n},
\end{equation}
which is the intersection of the partition cone with the plane at height $n$. With this notation, the restricted partition function is simply $p(n,3)=\#P(n,3)$.

One important property of the partition cone is that it has the dual description
\begin{eqnarray}
  \label{eqn:cone-v-description}
   \CCC_3 = V_3 \left(\R_{\geq0}^{2}\times\R_{>0}^{1}  \right) &\text{ where }& V_3 =  \mmat{ 6 & 3 & 2 \\ 0 & 3 & 2 \\ 0 & 0 & 2 }.
\end{eqnarray}
The columns of $V_3$ are called the \emph{generators} of the cone $\CCC_3$.
Equation (\ref{eqn:cone-v-description}) states that $\CCC_3$ is the set of all vectors
that can be written
as a linear combination of generators with non-negative real coefficients, where the coefficient of the last generator has to be strictly positive.
Note that $\CCC_3$ is \emph{simplicial}; i.e., the generators are linearly independent.
The generators are not uniquely determined; we could replace any generator $v$ by any positive multiple $\alpha v$ for $\alpha\in\R_{>0}$.
However, the generators in (\ref{eqn:cone-v-description}) have the crucial property that they all have integer components, and they are all at the same height.
Notice $6=\lcm{1,2,3}=:\lcm{3}$ is the lowest height at which this happens.
The idea of choosing generators in this way goes back to Ehrhart \cite{Ehrhartpolynomial,Ehrhart1,Ehrhart2}.

\begin{figure*}[t]
  \begin{subfigure}{0.48\textwidth}
  \centering
   \includegraphics[angle=270,width=6cm]{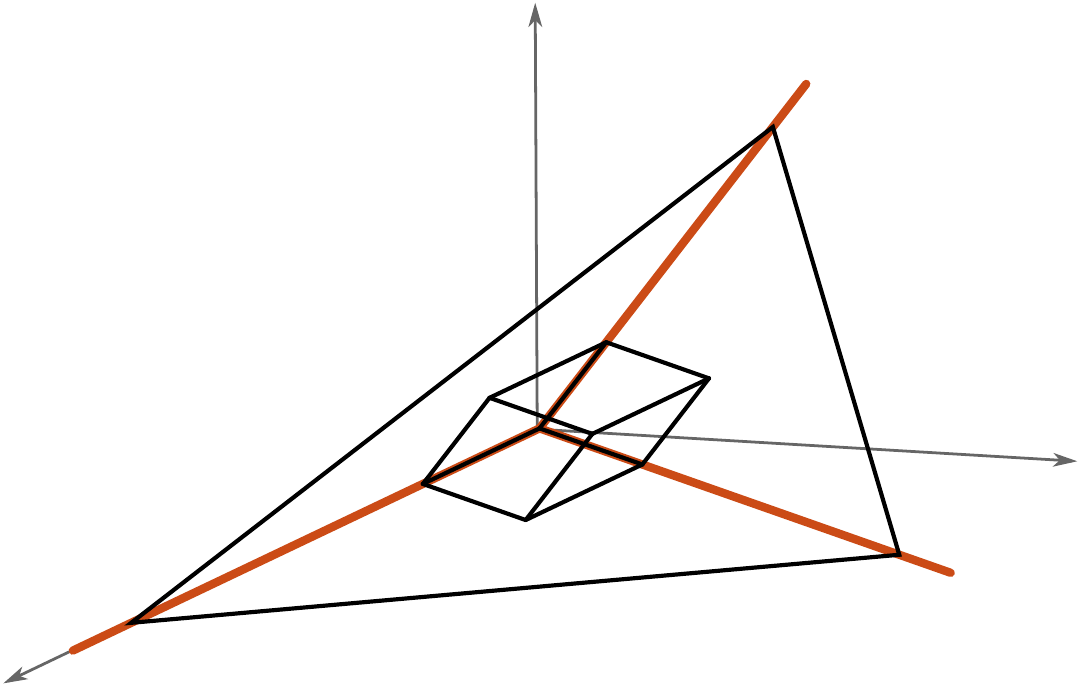}
   \caption{\label{fig:cone-tiling:l0} The fundamental parallelepiped $\FFF_3$ of $\CCC_3$.}
  \end{subfigure}
  \qquad
  \begin{subfigure}{0.48\textwidth}
  \centering
   \includegraphics[angle=270,width=6cm]{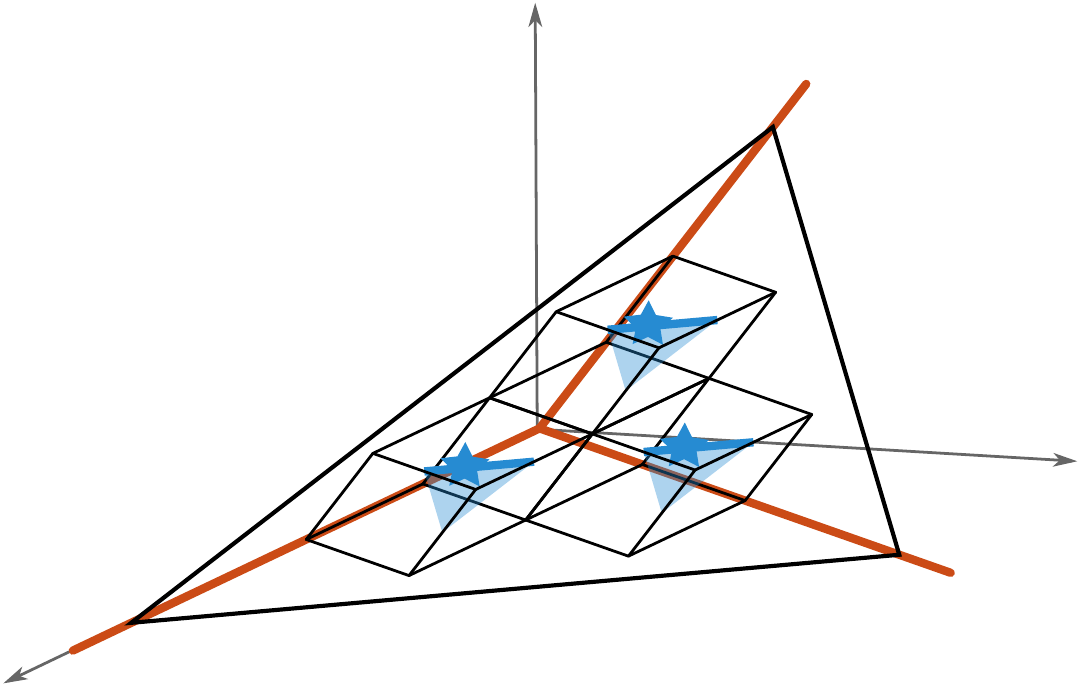}
   \caption{\label{fig:cone-tiling:l1} Translates of $\FFF_3$ by any one of the generators and their intersection with $\PPP(21,3)$. }
  \end{subfigure}

  \begin{subfigure}{0.48\textwidth}
  \centering
   \includegraphics[angle=270,width=6cm]{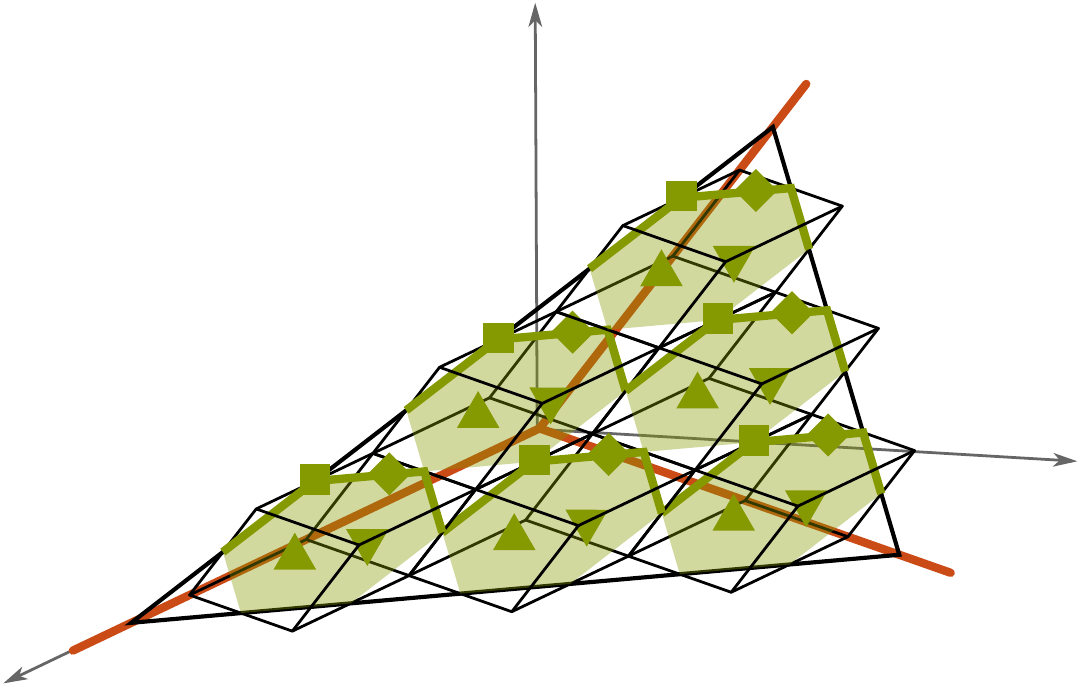}
   \caption{\label{fig:cone-tiling:l2} Translates of $\FFF_3$ by any two of the generators and their intersection with $\PPP(21,3)$. }
  \end{subfigure}
  \qquad
  \begin{subfigure}{0.48\textwidth}
  \centering
   \includegraphics[angle=270,width=6cm]{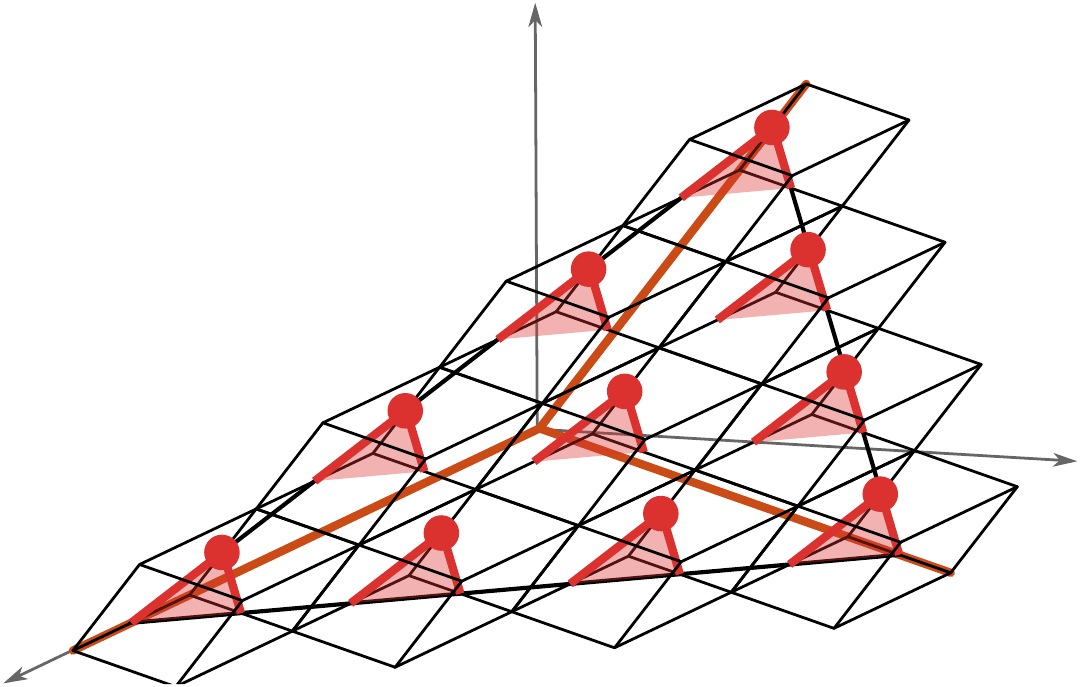}
   \caption{\label{fig:cone-tiling:l3} Translates of $\FFF_3$ by any three of the generators and their intersection with $\PPP(21,3)$.}
  \end{subfigure}
  \caption{ \label{fig:cone-tiling} The fundamental parallelepiped $\FFF_3$ tiles the cone $\CCC_3$, which induces a tiling of $\PPP(21,3)$ with slices of $\FFF_3$ at different heights.}
\end{figure*}

\begin{figure*}[t]
  \centering
   \includegraphics[angle=270,width=10cm]{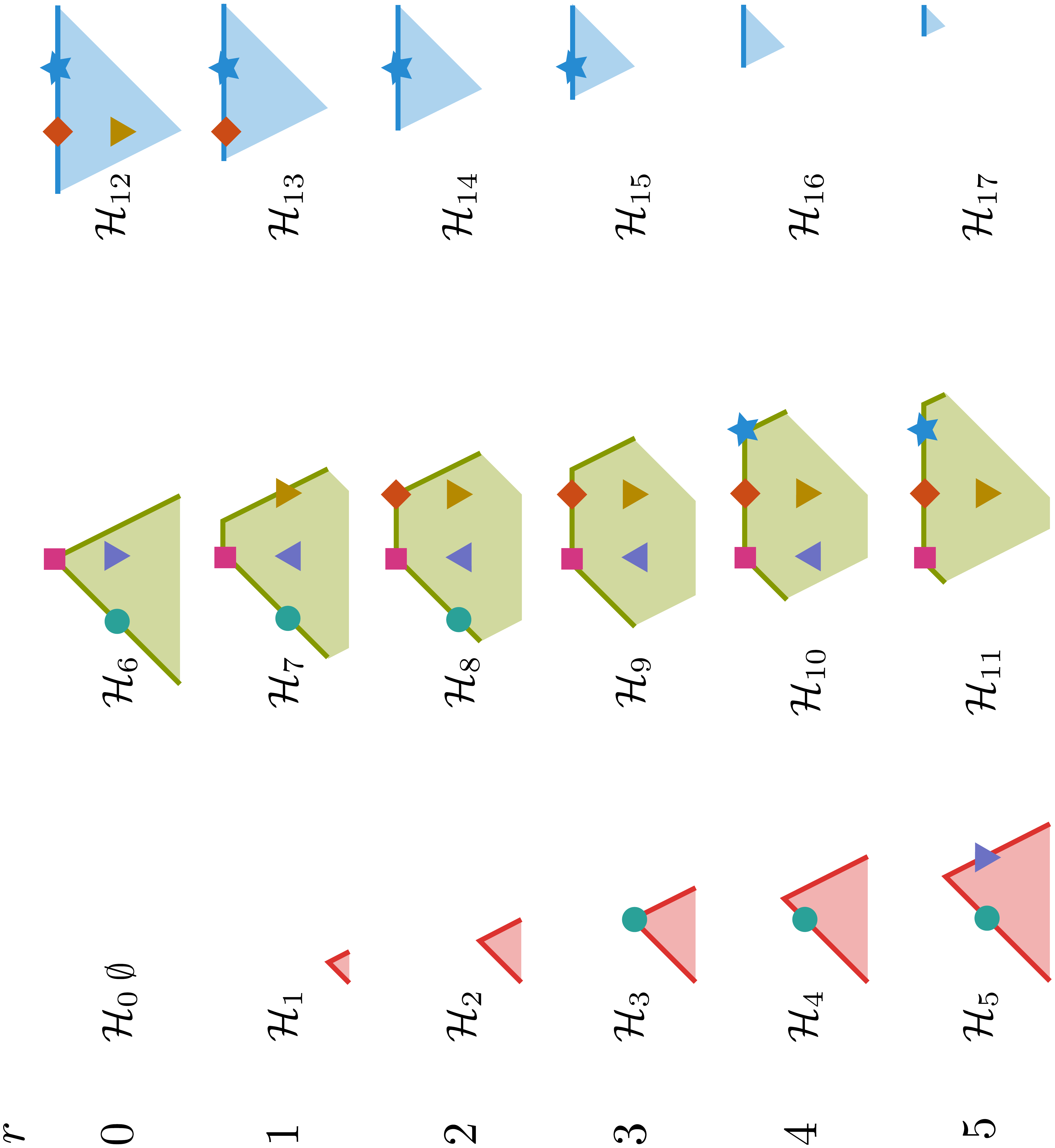}
  \caption[]{ \label{fig:master-tile-set} The slices of the fundamental parallelepiped: All the tiles $\HHH_i$ and the lattice points they contain. As we can see, there are 36 distinct lattice points in $\FFF_3$. The shapes of the lattice points $\lambda$ indicate the last two coordinates: Any lattice point $\lambda$ that is represented by a circle has last two coordinates $\lambda_2=1$ and $\lambda_3=1$. Going one step right increases $\lambda_2$ by 1. Going one step up increases $\lambda_3$ by 1. For example, the circle in $\HHH_4$ has coordinates $\msmat{2\\1\\1}$, the circle in $\HHH_8$ has coordinates $\msmat{6\\1\\1}$ and the diamond in $\HHH_{12}$ has coordinates $\msmat{7\\3\\2}$. }
\end{figure*}

As illustrated in Figure~\ref{fig:cone-tiling},
 (\ref{eqn:cone-v-description}) allows us to tile $\CCC_3$ with integer translates of the \emph{fundamental parallelepiped} $\FFF_3$, via
\begin{eqnarray}
  \label{eqn:cone-tiling}
   \CCC_3 = V_3\Z^3_{\geq0} + \FFF_3 &\text{ where }& \FFF_3 =  V_3\left( [0,1)^2 \times (0,1]^1 \right).
\end{eqnarray}
Here, $+$ denotes the Minkowski sum; i.e., $\CCC_3$ is obtained by translating $\FFF_3$ by every vector $v\in V_3\Z^3_{\geq0}$, and then taking the union, which is disjoint.
Again, we write $F_3 = \Z^3 \cap \FFF_3$. Note that the fundamental parallelepiped is not uniquely determined by the cone $\CCC_3$, but by our particular choice of generators.

For every $n$, the tiling (\ref{eqn:cone-tiling}) of $\CCC_3$ induces a tiling of the partition triangle $\PPP(n,3)$ as shown in Figure~\ref{fig:cone-tiling}.
To make this precise we define $\HHH_i=\FFF_3\cap\mset{\lambda\in\R^3}{|\lambda|=i}$ as the slice of $\FFF_3$ at height $i$,
and let $H_i=\Z^3\cap\HHH_i$ and $h^*_i = \#H_i$ denote the corresponding lattice point set and count, respectively.
All slices $\HHH_0$ through $\HHH_{17}$ and the lattice points they contain are shown in Figure~\ref{fig:master-tile-set}.
The coefficients in (\ref{p(6k,3)}) through (\ref{p(6k+5,3)}) reflect
the lattice point counts in the $\HHH_i$.

We define $T_k = \mset{v \in \Z^3_{\geq 0}}{ |v| = k }$ as the set of all non-negative integer vectors with coordinate sum $k$.
The elements of $T_k$ are necessarily arranged in a triangle pattern and are correspondingly counted by the triangular numbers $\#T_k = \binom{k+2}{2}$.
With this notation, the construction shown in Figure~\ref{fig:cone-tiling} then yields
\begin{eqnarray}
\PPP(6k+r,3) &=& \left( \HHH_{r} + V_3T_{k} \right) \cup \left( \HHH_{6+r} + V_3T_{k-1} \right) \cup \left( \HHH_{12+r} + V_3T_{k-2}  \right), \label{eqn:tiling}\\
P(6k+r,3)    &=& \left( H_{r} + V_3T_{k} \right) \cup \left( H_{6+r} + V_3T_{k-1} \right) \cup \left( H_{12+r} + V_3T_{k-2}  \right), \label{eqn:bijection} \\
p(6k+r,3)    &=& h_{r}^*{k+2  \choose 2} + h_{6+r}^*{k+1  \choose 2} + h_{12+r}^*{k \choose 2}\label{eqn:bijection2},
\end{eqnarray}
for any $r=0,\ldots,5$ and any $k\in\Z_{\geq0}$.
Here, we have made crucial use of the fact that our chosen generators are all at the same height and are integral. In particular, (\ref{eqn:bijection}) shows that the coefficients $h^*_i$ count lattice points at a certain height in the fundamental parallelepiped, and that the binomial coefficients count translation vectors $v\in T_j$.
This was Ehrhart's key observation, whence the vector $h^*$ is called Ehrhart $h^*$-vector or Ehrhart $\delta$-vector.

\begin{figure*}[t]
  \begin{subfigure}{0.32\textwidth}
  \centering
   \includegraphics[angle=270,width=3.6cm]{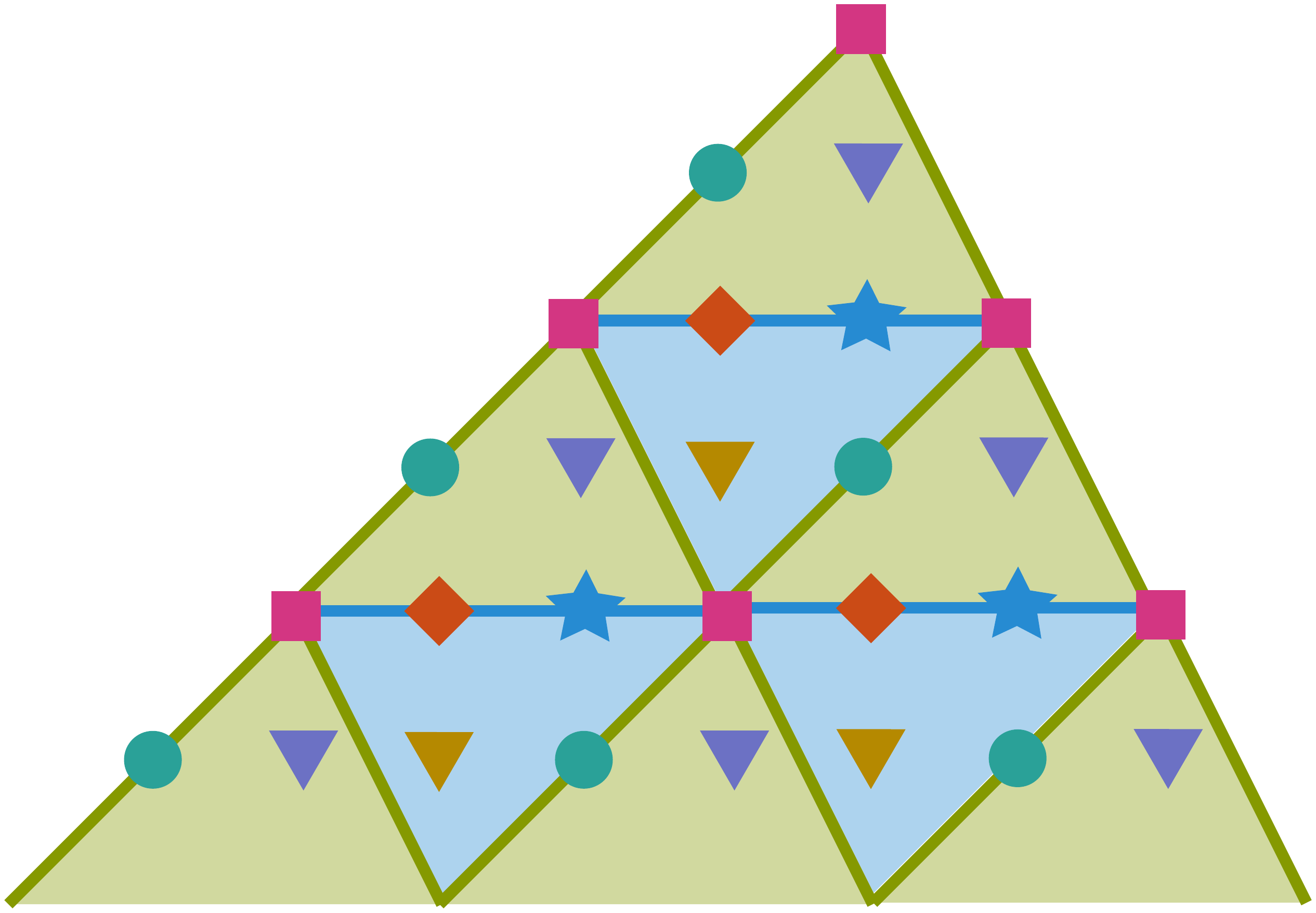}
   \caption{\label{fig:P18} $\PPP(18,3)$, $r=0$.}
  \end{subfigure}
  \begin{subfigure}{0.32\textwidth}
  \centering
   \includegraphics[angle=270,width=3.8cm]{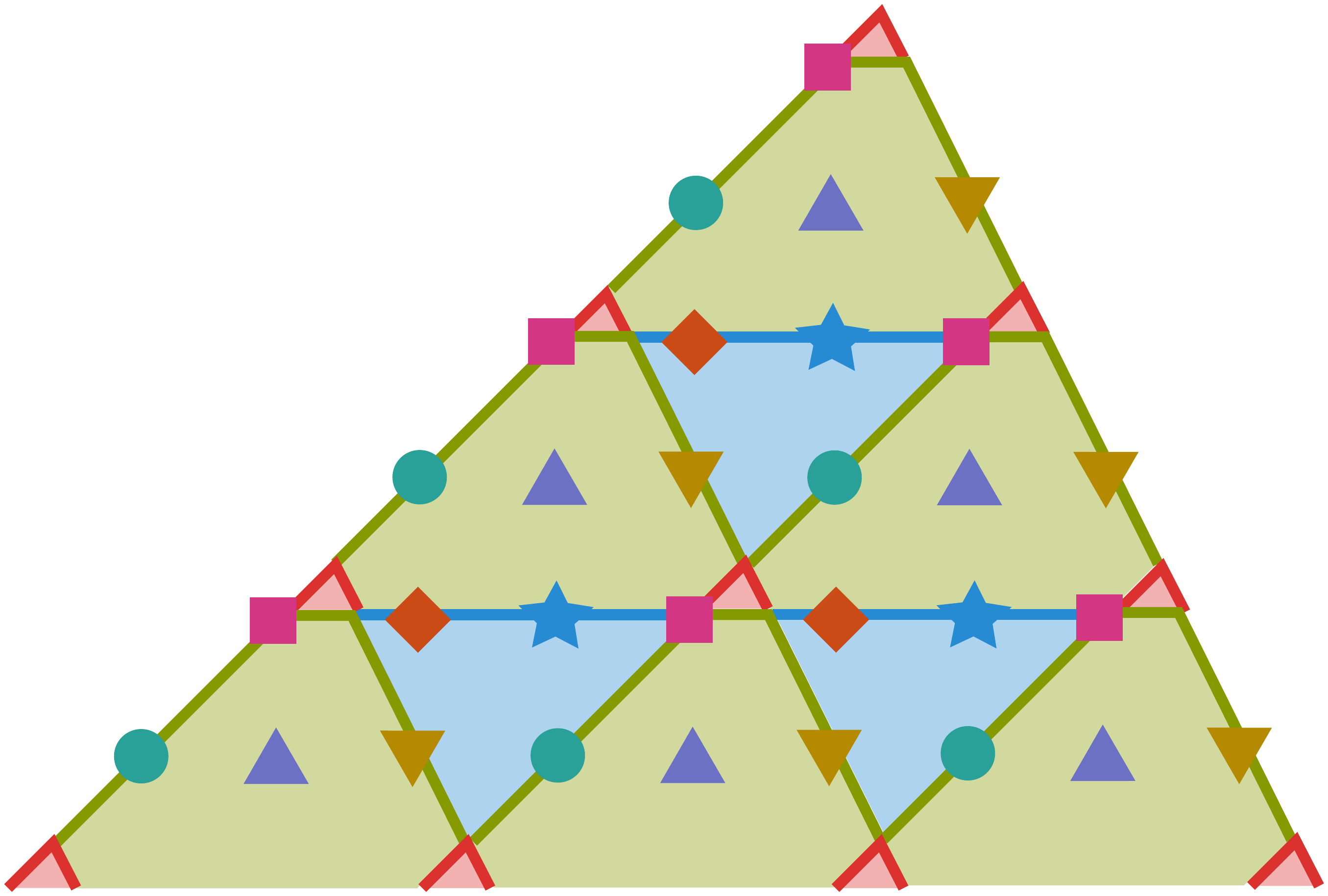}
   \caption{\label{fig:P19} $\PPP(19,3)$, $r=1$.}
  \end{subfigure}
  \begin{subfigure}{0.32\textwidth}
  \centering
   \includegraphics[angle=270,width=4cm]{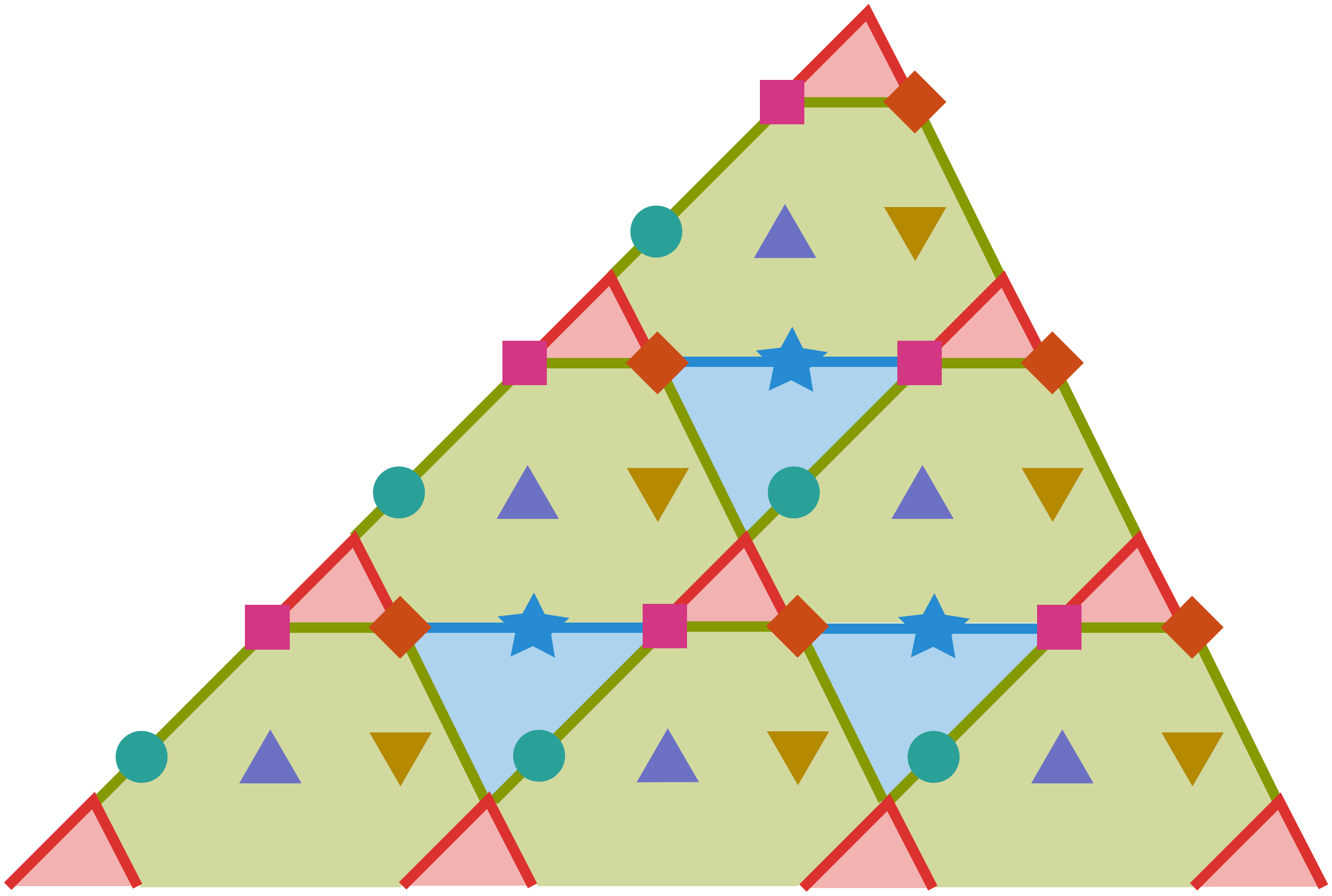}
   \caption{\label{fig:P20} $\PPP(20,3)$, $r=2$.}
  \end{subfigure}

  \begin{subfigure}{0.32\textwidth}
  \centering
   \includegraphics[angle=270,width=4.2cm]{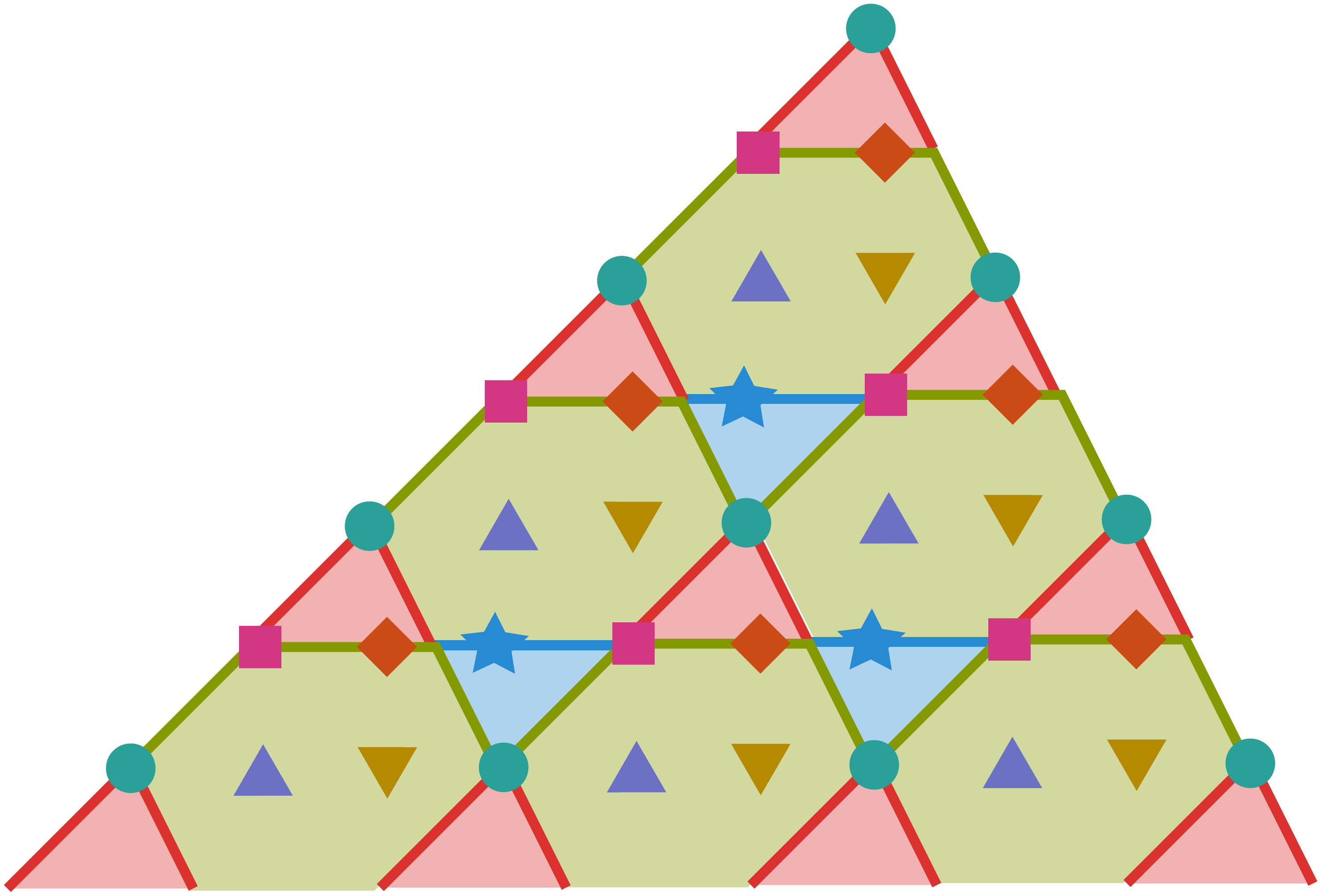}
   \caption{\label{fig:P21} $\PPP(21,3)$, $r=3$.}
  \end{subfigure}
  \begin{subfigure}{0.32\textwidth}
  \centering
   \includegraphics[angle=270,width=4.4cm]{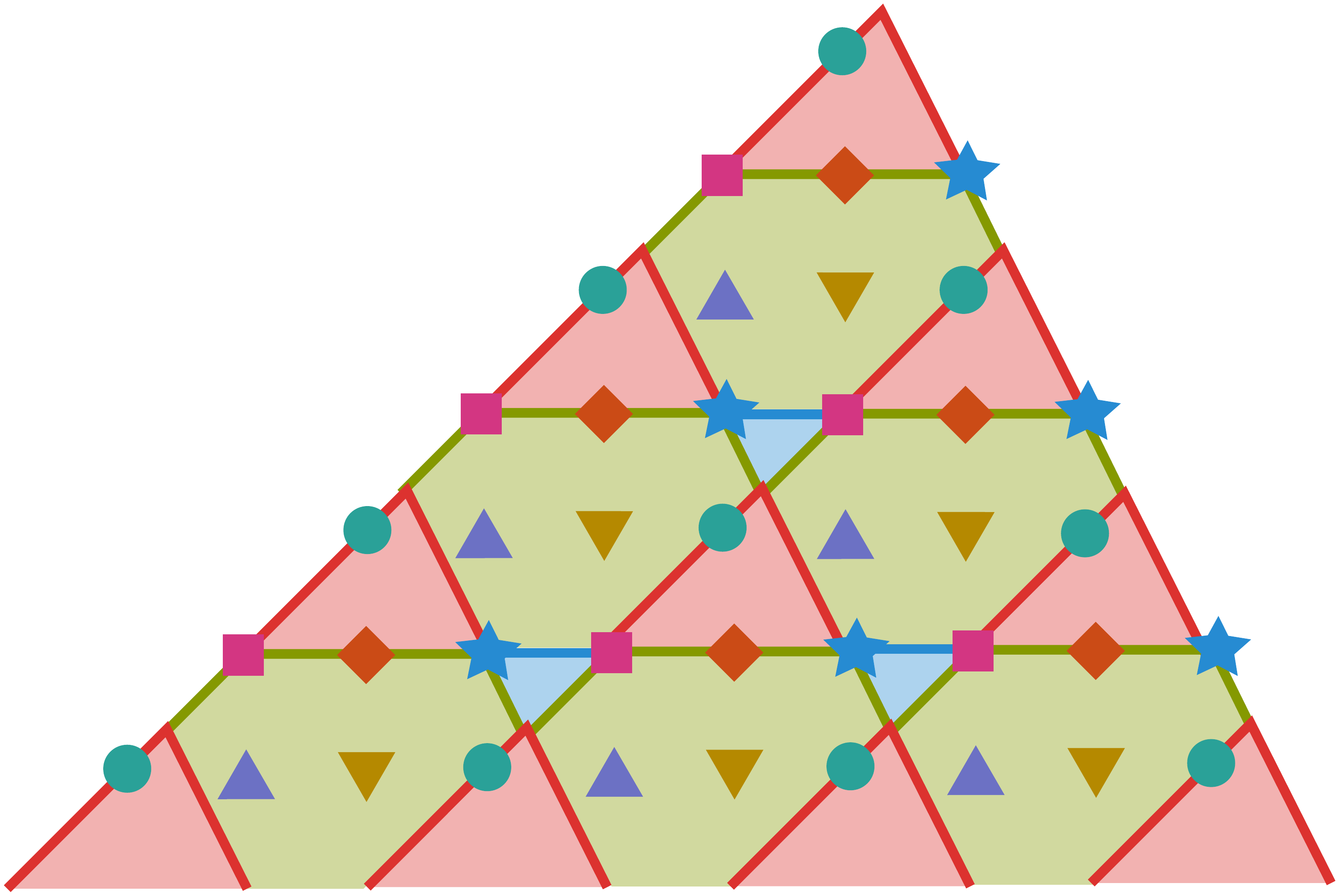}
   \caption{\label{fig:P22} $\PPP(22,3)$, $r=4$. }
  \end{subfigure}
  \begin{subfigure}{0.32\textwidth}
  \centering
   \includegraphics[angle=270,width=4.6cm]{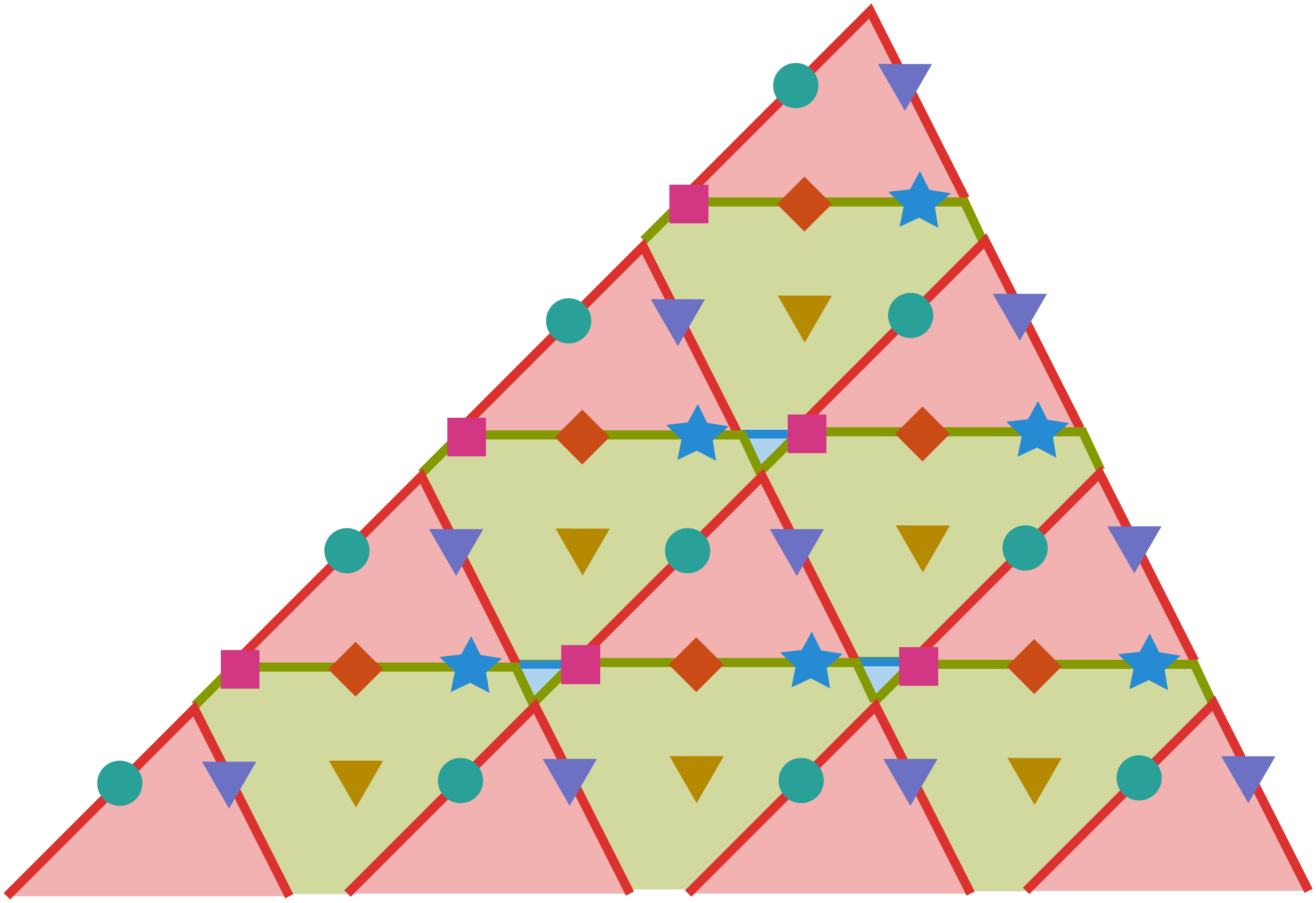}
   \caption{\label{fig:P23} $\PPP(23,3)$, $r=5$.}
  \end{subfigure}
  \caption{ \label{fig:tiling-examples} Tilings of $\PPP(18,3)$ through $\PPP(23,3)$.}
\end{figure*}

Figure~\ref{fig:tiling-examples} shows examples of the tiling given in (\ref{eqn:tiling}) for $\PPP(18,3)$ through $\PPP(23,3)$. Passing from the continuous tiles to lattice point sets yields (\ref{eqn:bijection}).
Counting these lattice points we obtain (\ref{eqn:bijection2}), which is the general form of equations (\ref{p(6k,3)}) through (\ref{p(6k+5,3)}).
This construction shows that the coefficients of (\ref{eqn:bijection2}), hence, (\ref{p(6k,3)}) through (\ref{p(6k+5,3)}) are precisely the lattice points counts in the $\HHH_i$ as is evident in Figure~\ref{fig:master-tile-set}.

Since all unions are disjoint, (\ref{eqn:bijection}) allows us decompose any partition $\lambda\in P(n,3)$ uniquely into a partition (the box remainder) $\mu\in F_3$ in the fundamental parallelepiped and a partition of the form $V_3\tau$, where $\tau\in\Z^3_{\geq 0}$ (the box quotient) is a non-negative integer vector.
For brevity, we will simply say that we decompose $\lambda$ into the pair $(\mu,\tau)$.
This decomposition is vital for our results,
and so we summarize it in the following definition and lemma.
See Figure~\ref{fig:decomposition} for an illustration.

\begin{figure*}[t]
  \begin{center}
  \includegraphics[angle=270,width=14cm]{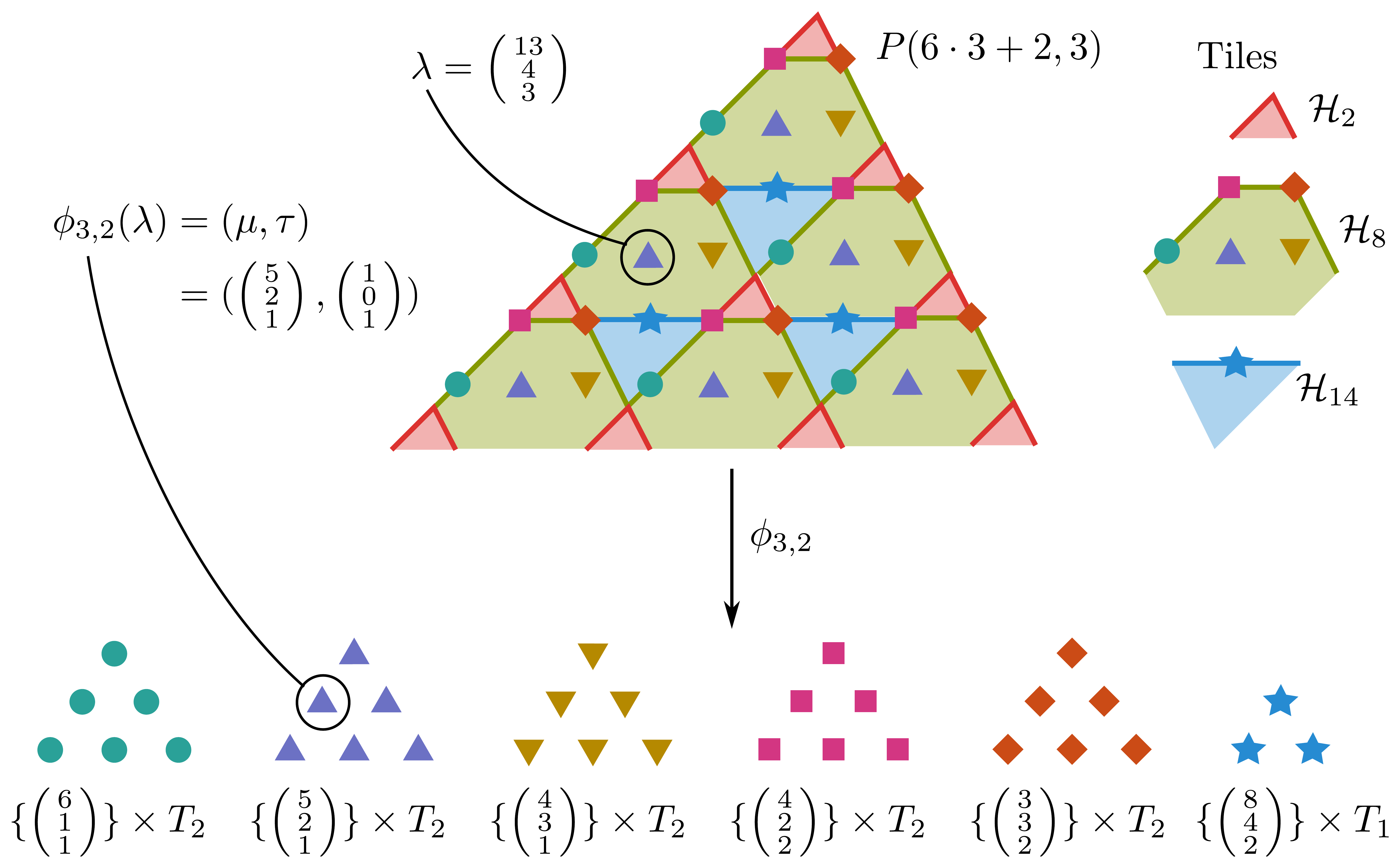}
  \end{center}
  \caption[]{ \label{fig:decomposition} The partition triangle $\PPP(20,3)$ is tiled with certain integer translates of $\HHH_2,\HHH_8,\HHH_{14}$. In other words, the set $P(20,3)$ is decomposed into zero triangles $T_3$ (as $H_2$ is empty), 5 triangles $T_2$ (one for each lattice point in $H_8$) and 1 triangle $T_1$ (one for each lattice point in $H_{14}$).  Each partition $\lambda$ in $P(20,3)$ is decomposed into a fundamental partition $\mu$ and a non-negative integer vector $\tau$ via $\lambda=\mu+V_3\tau$. For example, $\lambda=\msmat{13\\4\\3}=\msmat{5\\2\\1}+V_3 \msmat{1\\0\\1}$ so that $\phi_{3,2}(\lambda)=(\mu,\tau)$ with
  $\mu=\msmat{5\\2\\1}$ and $\tau=\msmat{1\\0\\1}$.}

\end{figure*}

\begin{definition}
\label{def:decomposition}
For all $k\in \Z_{\geq 0}$, $r\in\{0,\ldots,5\}$, and $\lambda \in P(6k+r,3)$, write
\begin{eqnarray}
\label{eqn:box-decomposition}
  \lambda = \mu + V_3\tau \;\;\;\; &\text{ such that }& \;\;\; \mu \in F_3   \text{\ \  and \ }  \tau\in\Z^3_{\geq 0}.
\end{eqnarray}
Define
\begin{eqnarray*}
  \phi_{k,r} \; : \; P(6k+r,3) &\longrightarrow& \left( H_{r} \times T_k \right) \cup \left( H_{6+r} \times T_{k-1} \right) \cup \left( H_{12+r} \times T_{k-2} \right) =: Q_{k,r}\\
\text{by} \qquad  \lambda &\mapsto& (\mu,\tau).
\end{eqnarray*}
We abbreviate $\phi=\bigcup_{k,r}\phi_{k,r}$.
The pair $(\mu,\tau)$ is called the \emph{box decomposition} of $\lambda$.
The partition $\mu\in F_3$ is the \emph{box remainder} and the vector $\tau$ is the \emph{box quotient} of $\lambda$.
\end{definition}

As we have proven above, the crucial property of $\phi_{k,r}$ is the following.

\begin{lemma}
\label{lem:decomposition}
For all $k\in \Z_{\geq 0}$ and $r\in\{0,\ldots,5\}$, the map
\begin{eqnarray*}
  \phi_{k,r} \; : \; P(6k+r,3) &\longrightarrow&  Q_{k,r}
\end{eqnarray*}
is a bijection.
\end{lemma}
As (\ref{eqn:box-decomposition}) suggests, we can think of the box decomposition as a division with remainder. In addition to the example given in Figure~\ref{fig:decomposition}, consider the box decomposition
\[
  \msmat{11\\5\\2} = \msmat{2\\2\\2} + V_3 \msmat{1\\1\\0}.
\]
Even though $\msmat{11\\5\\2} = V_3 \msmat{1\\1\\1}$, we have $\tau = \msmat{1\\1\\0}$ since $\msmat{2\\2\\2}\in F_3$ but $\msmat{0\\0\\0}\not\in F_3$.  The underlying reason is that we are working with partitions with \emph{exactly} three parts. In fact, $\msmat{2\\2\\2}$ is the magenta square in $\HHH_6$ as shown in Figure~\ref{fig:master-tile-set}.

The box decomposition has a very nice combinatorial interpretation in terms of a decomposition of the Ferrers diagram of $\lambda$
into boxes, which is explored in \cite{BreuerEichhornKronholm-forthcoming}. For the purposes of this paper our focus is on the geometric point of view.


\section{Ehrhart cranks for witnessing divisibilities}
\label{sec:Ehrhart-cranks}

In this section, we construct a whole family of combinatorial witnesses for each and every arithmetic progression
of divisibility for $p(n,3)$ modulo any prime $m \equiv -1 \pmod 6$.
The basic strategy is the following:
The decomposition $\phi_{k,r}$ from Definition~\ref{def:decomposition} allows us to divide the set of partitions $P(6k+r,3)$ into $h^*_r$ copies of the triangle $T_k$, $h^*_{6+r}$ copies of the triangle $T_{k-1}$, and $h^*_{12+r}$ copies of the triangle $T_{k-2}$.
Any way in which these triangles can be reassembled into a rectangle
of lattice points wherein the number of lattice points along the width or height of the rectangle is
divisible by $m$ will provide us with cycles proving divisibility and
a crank statistic that witnesses divisibility for that fixed $r$.
Here it is crucial that for a given $r$ there is an arrangement of triangles that works for all $k$, so that we do indeed obtain a crank function that witnesses divisibility for the entire arithmetic progression with a fixed remainder.

We first
give an informal statement of the theorem,
and then in Theorem~\ref{thm:unlabeled-technical}, we give the technical details.

\begin{theorem}\label{InformalTheorem}
For fixed $r'$ and $m$, and any $k'\geq 0$, let $r,k$ be such that $6mk' + r' = 6k+r$. Then, any way $\psi$ of reassembling $h^*_{r}$ triangles $T_k$, $h^*_{r+6}$ triangles $T_{k-1}$, and $h^*_{r+12}$ triangles $T_{k-2}$ into a rectangle $R$ with one side divisible by $m$ defines a crank function $c=\eta\circ\psi\circ\phi$, as illustrated in Figure~\ref{fig:unlabeled-2m-2}, that witnesses the congruence
\begin{eqnarray}
  \label{eqn:unlabeled-congruence-motivation}
  p(6mk'+r',3) \equiv 0 \pmod m \text{ for all }k',
\end{eqnarray}
where cycles are given by traversing either the rows or columns of $R$.
\end{theorem}

\begin{figure*}[t]

  \begin{center}
  \includegraphics[angle=0,width=12cm]{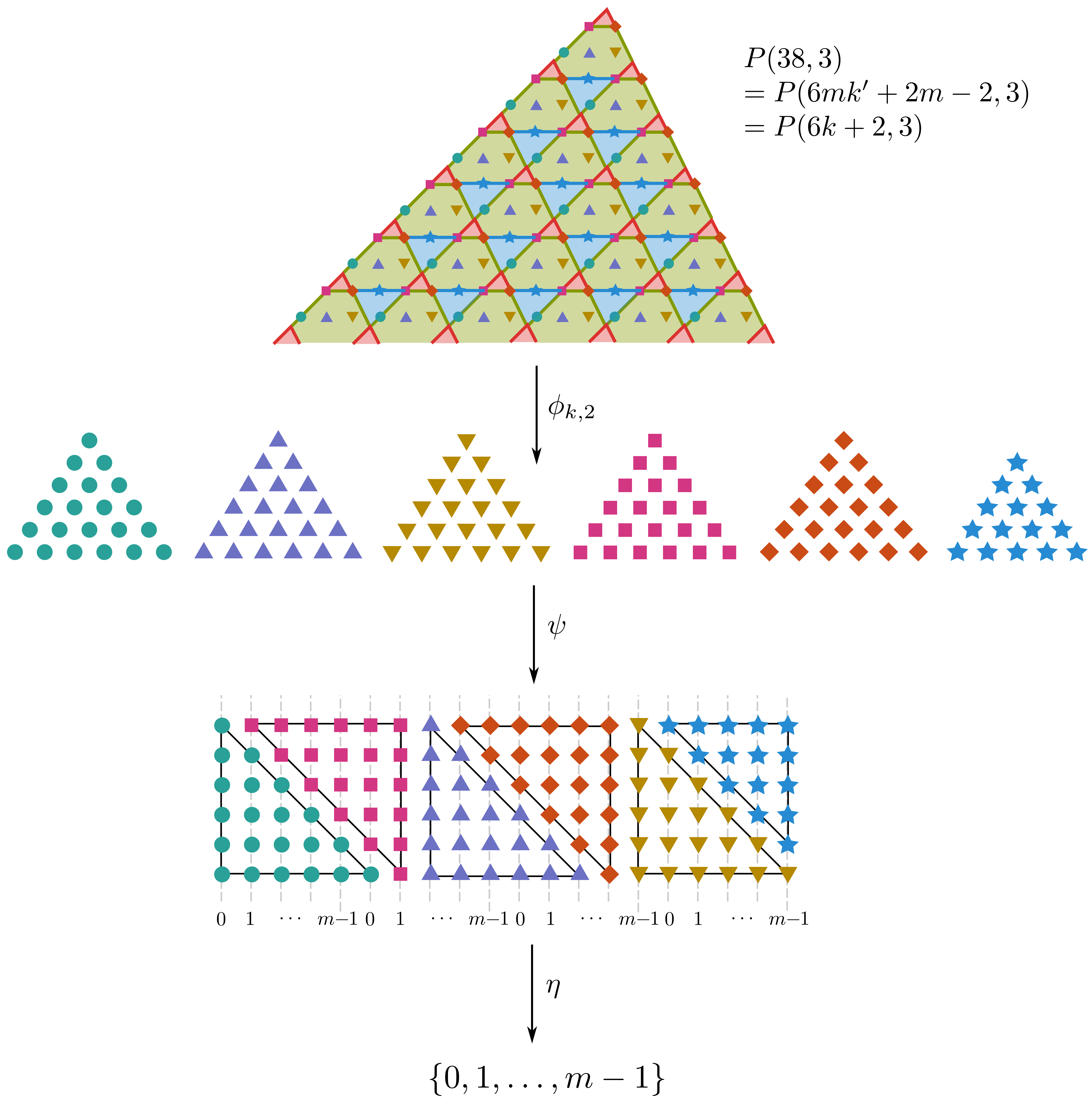}
  \end{center}

  \caption{ \label{fig:unlabeled-2m-2} The construction of an Ehrhart crank $c=\eta\circ\psi\circ\phi$ for the case $r'=2m-2$. The lattice point sets are shown for $m=5, k'=1, k=mk'+\frac{m-2}{3}=6$. However, the rectangle arrangement given by $\psi$ is independent of these parameters.}
\end{figure*}

To work our way up to a precise technical statement of this result, let us consider the case of remainder $r'=2m -2$ as an example, as shown in Figure~\ref{fig:unlabeled-2m-2}. Our goal is to witness (\ref{eqn:unlabeled-congruence-motivation}). To begin, $\phi$ is given by the decomposition from Section~\ref{sec:ehrhart}. To this end, we rewrite
\[
  6mk' + 2m -2 = 6\left(mk'+\frac{m-2}{3}\right) + 2,
\]
and let $K=\mset{mk' + \frac{m-2}{3}}{k' \geq 0}$.
By (\ref{eqn:bijection}) and (\ref{eqn:bijection2}), we have for all $k\in K$,
\begin{eqnarray*}
  p(6k+2,3) &=& 0 \binom{k+2}{2} + 5 \binom{k+1}{2} + 1 \binom{k}{2} \qquad \text{and} \\
  P(6k+2,3) &=& (H_2 + V_3 T_k) \cup (H_8 + V_3 T_{k-1}) \cup (H_{14} + V_3 T_{k-2}),
\end{eqnarray*}
where $H_2=\emptyset$, $H_8=\{\mu_{8,1},\mu_{8,2},\mu_{8,3},\mu_{8,4},\mu_{8,5}\}$, $H_{14}=\{\mu_{14,1}\}$ and
\begin{eqnarray*}
\mu_{8,1} = \msmat{6\\1\\1},
\mu_{8,2} = \msmat{4\\2\\2},
\mu_{8,3} = \msmat{5\\2\\1},
\mu_{8,4} = \msmat{3\\3\\2},
\mu_{8,5} = \msmat{4\\3\\1},
\mu_{14,1} = \msmat{8\\4\\2}.
\end{eqnarray*}
Thus $\phi_{k,2}$ gives a bijection between $P(6k+2,3)$ and $Q_{k,2}$ for all $k\in K$. In Figure~\ref{fig:unlabeled-2m-2}, the set $Q_{k,2}$ is visualized as 6 disjoint triangles, as explained in Section~\ref{sec:ehrhart}.

We now show that the six interlacing triangles in $P(6k+2,3) = (H_2 + V_3 T_k) \cup (H_8 + V_3 T_{k-1}) \cup (H_{14} + V_3 T_{k-2})$ can be rearranged to make a rectangle having at least one side of length a multiple of $m$.
This is done by noting that $Q_{k,2}$, being the six triangles indicated by the binomial coefficients in $ 5 \binom{mk'+\frac{m-2}{3}+1}{2} + 1 \binom{mk'+\frac{m-2}{3}}{2} $, can be arranged into a rectangle $R(k')$ of width $\ell_1(k') = m(3k'+1)$ and height $\ell_2(k') =mk' + \frac{m-2}{3}$.
Formally, we let $\RRR^{\ell_1(k')}_{\ell_2(k')}=[0,\ell_1(k')-1]\times[0,\ell_2(k')-1]$ denote the continuous rectangle,
and write $R(k'):=R^{\ell_1(k')}_{\ell_2(k')} = \Z^2 \cap \RRR^{\ell_1(k')}_{\ell_2(k')}$ for the set of lattice points therein.
Since, in this example, $\ell_1(k')$ is a multiple of $m$ for all $k'$, it follows that $\#R(k')$ is divisible by $m$, and this divisibility is witnessed by the crank
\begin{eqnarray*}
  \eta : \R^2 &\rar& \{0,\ldots,m-1\}\\
    (x,y) &\mapsto& x \mmod m.
\end{eqnarray*}

The cycles that go along with this crank are simply moving along a row of the rectangle. Our strategy will therefore be to reassemble the triangles in $Q_{k,2}$ to form the rectangle $R(k')$.
To this end, assume that for every $i=1,\ldots,6$ we have an affine linear function $\psi_{\mu_i}:\R^3\rar \R^2$ such that for every $s \in \Z_{\geq 1}$, the restriction
$\psi_{\mu_i}|_{T_s}$ is injective, and
\[
  \psi_{\mu_1}(T_{k-1}) \cup \psi_{\mu_2}(T_{k-1}) \cup \psi_{\mu_3}(T_{k-1}) \cup \psi_{\mu_4}(T_{k-1}) \cup \psi_{\mu_5}(T_{k-1}) \cup \psi_{\mu_6}(T_{k-2}) = R(k')
\]
for all $k\in K$, where the unions are disjoint and $k'$ is determined by $k=mk'+\frac{m-2}{3}$.
Then, the function
\begin{eqnarray*}
  \psi : Q_{k,r} &\rar& R(k')\\
    (\mu,\tau) &\mapsto& \psi_{\mu}(\tau)
\end{eqnarray*}
is a bijection. Affine linear functions with this property can be understood by an inspection of Figure~\ref{fig:unlabeled-2m-2}.
We will give explicit formulas for this example below, but first, let us reap the benefits of this construction: the composition $c=\eta\circ\psi\circ\phi_{k,2}$,
\begin{align*}
  P(6mk'+r') \;\; & \xrightarrow{\phi_{k,2}} && Q_{k,2} && \xrightarrow{\psi} && R(k') && \xrightarrow{\eta} && \{0,\ldots,m-1\} \\
  \lambda \;\; & \mapsto && (\mu,\tau) && \mapsto && (x,y) && \mapsto && x \mmod m,
\end{align*}
is a crank function witnessing the congruence (\ref{eqn:unlabeled-congruence-motivation}) for $r'=2m-2$. Cycles are defined by moving along a row of $R(k')$. One important aspect here is that the formulas for $\phi,\psi,\eta$ do not depend on $k$. In this sense, $c$ is one function that witnesses divisibility for the entire congruence (\ref{eqn:unlabeled-congruence-motivation}) for fixed $r'=2m-2$.

The above construction generalizes, which proves the following theorem.

\begin{theorem}
\label{thm:unlabeled-technical}
Let $m\in\Z_{\geq 1}$ and $r'\in\Z_{\geq0}$ be fixed. Let $r\in \{0,\ldots,5\}$ and $k:\Z\rar\Z$ such that $6mk' + r' = 6k(k')+r$ for all $k'\in \Z_{\geq 0}$. Let $K=\mset{k(k')}{k'\in \Z_{\geq0}}$. Let $\ell_1,\ell_2:\Z\rar\Z$ and $R(k')=R^{\ell_1(k')}_{\ell_2(k')}$.
Let $H = H_r \cup H_{r+6} \cup H_{r+12}$.
For each $\mu\in H_{r+6i} \subset H$, let $\psi_\mu:\R^3\rar\R^2$ be an affine linear function such that for every $s \in \Z_{\geq 1}$, the restriction
$\psi_{\mu_i}|_{T_s}$ is injective, and
\[
  \bigcup_{\mu \in H_{r+6i} \subset H} \psi_{\mu}(T_{k(k')-i}) = R(k')
\]
for all $k'\in \Z_{\geq 0}$, where the union is disjoint. Let $\psi:Q_{k(k'),r}\rar R(k')$ be defined by $\psi(\mu,\tau)=\psi_\mu(\tau)$. Let $\eta:\R^2\rar\{0,\ldots,m-1\}$ be a function of the form $\eta(x,y)=\alpha x + \beta y + \gamma \mmod m$ for $\alpha,\beta,\gamma\in\Z$. If at least one of the following two conditions hold:
\begin{enumerate}[label=(\roman*), ref=(\roman*)]
\item \label{enum:unlabeled-width} $\gcd(\alpha,m)=1$ and $m|\ell_1(k')$ for all $k'$, or
\item \label{enum:unlabeled-height} $\gcd(\beta,m)=1$ and $m|\ell_2(k')$ for all $k'$,
\end{enumerate}
then the composition $c=\eta\circ\psi\circ\phi$ is a crank function witnessing the congruence
\[
  p(6mk'+r',3) \equiv 0 \pmod m.
\]
In case \ref{enum:unlabeled-width}, let $\delta=\msmat{1\\0}$. In case \ref{enum:unlabeled-height}, let $\delta=\msmat{0\\1}$. Cycles of length $\ell_1(k')$ or $\ell_2(k')$, respectively, are then given by the operation
\[
  \lambda \longmapsto (\psi\circ\phi)^{-1}((\psi\circ\phi)(\lambda) + \delta),
\]
where the addition is taken modulo $\ell_1(k')$ or $\ell_2(k')$, respectively.
\end{theorem}

Such rectangles $R$ and triangle arrangements $\psi$ exist for all the cases covered in Proposition~\ref{6j-1prop}. We call crank functions of the form $c=\eta\circ\psi\circ\phi$ as given in Theorem~\ref{thm:unlabeled-technical} \emph{Ehrhart cranks}.

\begin{lemma}
  \label{lem:unlabeled-arrangements}
  For all $r'\in\{\pm0,1,2,2m-2,2m+1\}$ there exist lengths $\ell_1,\ell_2$ and affine linear functions $\psi_\mu$ satisfying the conditions of Theorem~\ref{thm:unlabeled-technical}.
\end{lemma}

\begin{table}
\begin{center}
\begin{tabular}{ l | l | l | l | l | l | l | l}
$r'$       &  $k$                   & $r$ & $\ell_1$   & $\ell_2$              & $h^*_r$ & $h^*_{r+6}$ & $h^*_{r+12}$ \\
\hline
$-(2m+1)$  &   $mk'-\frac{m+1}{3}$      &  1   &   $m(3k'-1)$   &    $mk'-\frac{m+1}{3}$   &    0   &    4   & 2           \\
$-(2m-2)$  &   $mk'-\frac{m+1}{3}$     &   4  &    $m(3k'-1)$        &     $mk'-\frac{m-2}{3}$ &   1  &  5   &   0         \\
$-2$       &      $mk'-1$    &   4  &      $3mk'-2$      &    $mk'$      &    1   &      5     &     0       \\
$-1$       &   $mk'-1$    &   5  & $3mk'-1$ &    $mk'$    & 2      &    4       &     0       \\
$0$        &       $mk'$       &  0   &  $3mk'$    &   $mk'$       &  0     &     3      &  3          \\
$1$        &         $mk'$     &  1   &   $3mk'+1$   &    $mk'$ &   0    &          4 &     2       \\
$2$        &     $mk'$   &  2   &   $3mk'+2$   &   $mk'$   &    0   &        5   &      1      \\
$2m-2$     &  $mk' + \frac{m-2}{3}$ & $2$ & $m(3k'+1)$ & $mk' + \frac{m-2}{3}$ & 0   & 5       & 1        \\
$2m+1$     &   $mk' + \frac{m-2}{3}$   &  5   &  $m(3k'+1)$  &   $mk'+\frac{m+1}{3}$  &  2     &     4      & 0
\end{tabular}
\end{center}

\caption{\label{tab:unlabeled} Parameters of Theorem~\ref{thm:unlabeled-technical} for every $r'\in\{\pm0,1,2,2m-2,2m+1\}$.}
\end{table}

\begin{figure*}[t]

  \begin{center}
  \includegraphics[angle=270,width=14.5cm]{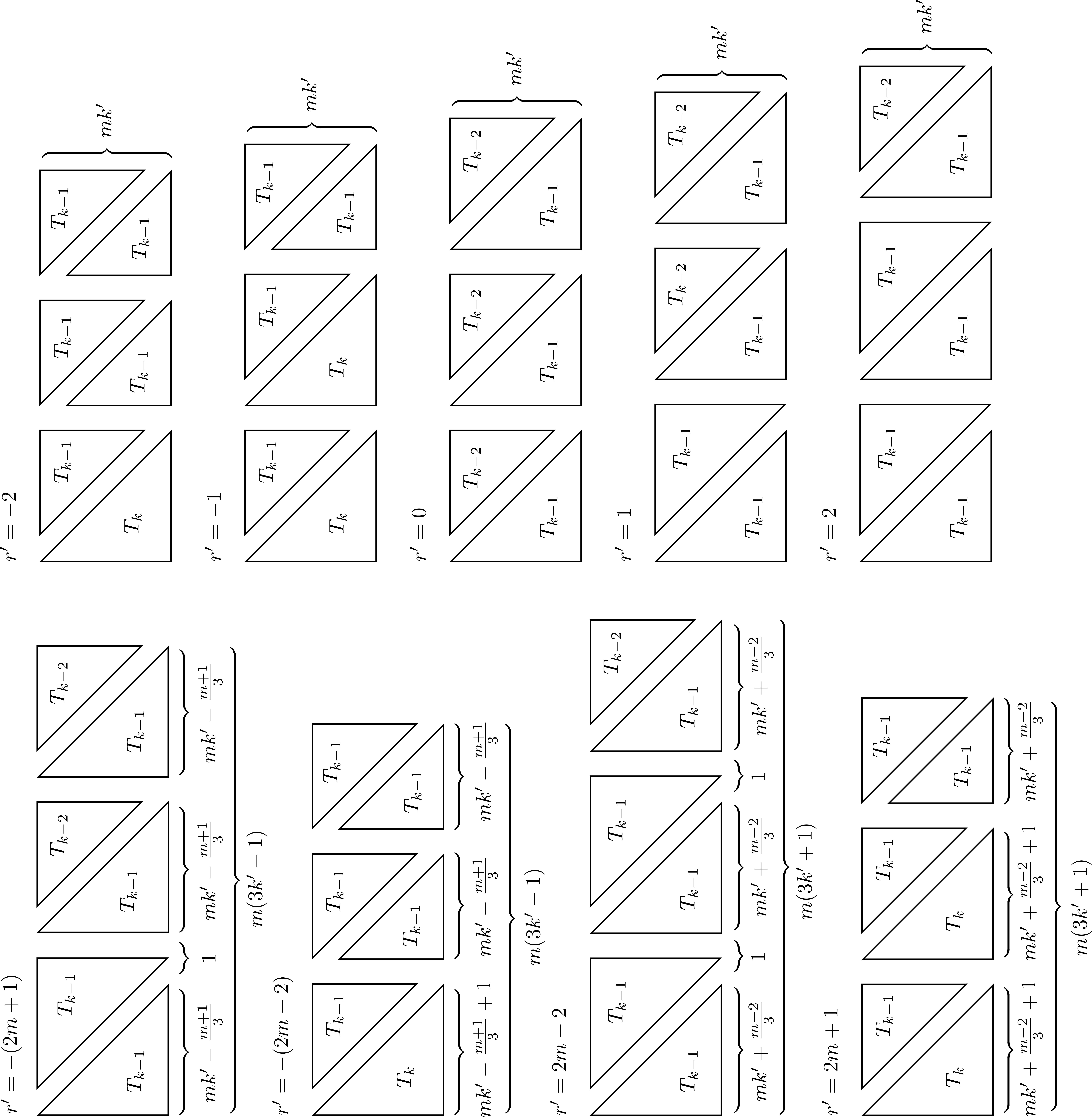}
  \end{center}

  \caption{ \label{fig:unlabeled-cases} The different triangle arrangements $\psi$ for every $r'\in\{\pm0,1,2,2m-2,2m+1\}$.}
\end{figure*}

\begin{proof}
  For every $r'$, the values of $k,r,\ell_1,\ell_2$ and $h^*_{r},h^*_{r+6},h^*_{r+12}$ are given in Table~\ref{tab:unlabeled} (see also Figure~\ref{fig:master-tile-set}).
  We work with $\eta(x,y) = x \mmod m$ in the cases $r'\in{\pm(2m-2),\pm(2m+1)}$ as here the width of $R$ is divisible by $m$, and with $\eta(x,y) = y \mmod m$ in the cases $r\in\{0,\pm1,\pm2\}$ as here the height of $R$ is divisible by $m$. The triangle arrangements $\psi$ are given in Figure~\ref{fig:unlabeled-cases}.
\end{proof}

\begin{corollary}
There exist Ehrhart cranks with explicit cycles providing combinatorial witnesses for the congruences of $p(n,3)$ as characterized by Proposition~\ref{6j-1prop}.
\end{corollary}

\begin{proof}
Follows from Theorem~\ref{thm:unlabeled-technical} and Lemma~\ref{lem:unlabeled-arrangements}.
\end{proof}

Figure~\ref{fig:unlabeled-cases} gives a clear picture of how the maps $\psi_\mu$ are to be constructed. For instructive purposes we work out an explicit formula for $c=\eta\circ\psi\circ\phi$ as given in Theorem~\ref{thm:unlabeled-technical} for the case $r'=2m-2$ below.

The triangle arrangement from Figures~\ref{fig:unlabeled-2m-2} and \ref{fig:labeled-origami-cases-horizontal} is given by the six affine linear functions:
\begin{eqnarray*}
\psi_{\msmat{6\\1\\1}}(\tau) &=& \msmat{0 & 1 & 0 \\ 0 & 0 & 1}\tau \\
\psi_{\msmat{4\\2\\2}}(\tau) &=& \msmat{1 & 1 & 0 \\ 0 & 1 & 1}\tau + \msmat{1\\0}\\
\psi_{\msmat{5\\2\\1}}(\tau) &=& \msmat{0 & 1 & 0 \\ 0 & 0 & 1}\tau + \msmat{1 & 1 & 1 \\ 0 & 0 & 0}\tau + \msmat{2\\0}\\
\psi_{\msmat{3\\3\\2}}(\tau) &=& \msmat{1 & 1 & 0 \\ 0 & 1 & 1}\tau + \msmat{1 & 1 & 1 \\ 0 & 0 & 0}\tau + \msmat{3\\0}\\
\psi_{\msmat{4\\3\\1}}(\tau) &=& \msmat{0 & 1 & 0 \\ 0 & 0 & 1}\tau + \msmat{2 & 2 & 2 \\ 0 & 0 & 0}\tau + \msmat{4\\0}\\
\psi_{\msmat{8\\4\\2}}(\tau) &=& \msmat{1 & 1 & 0 \\ 0 & 1 & 1}\tau + \msmat{2 & 2 & 2 \\ 0 & 0 & 0}\tau + \msmat{5\\1}.
\end{eqnarray*}
Exploiting the pattern evident in these formulas, $\psi(\mu,\tau)$ can be written more compactly as
\begin{eqnarray*}
  \psi(\mu,\tau) =
  \msmat{
    \tau_2 + (\bar{\mu}_3-1)(\tau_1+1) + \bar{\mu}_2(2+\tau_1+\tau_2+\tau_3) \\
    \tau_3 + (\bar{\mu}_3-1)\tau_2 + \frac{\mu_1 + \mu_2 + \mu_3 - 8}{6}
  },
\end{eqnarray*}
where $\bar{\mu}$ is a vector encoding the multiplicities of the columns of the Ferrers diagram of the fundamental partition $\mu$; i.e.,
\[
  \bar{\mu} = \msmat{1&1&1\\0&1&1\\0&0&1}^{-1} \cdot\mu ~~=~~ \msmat{\mu_1 - \mu2 \\ \mu_2 - \mu_3 \\ \mu_3},
\]
which in the case of the triangle arrangement we chose for $r'=2m-2$ gives
\[
\bar{\msmat{6\\1\\1}} = \msmat{5\\0\\1},
\bar{\msmat{4\\2\\2}} = \msmat{2\\0\\2},
\bar{\msmat{5\\2\\1}} = \msmat{3\\1\\1},
\bar{\msmat{3\\3\\2}} = \msmat{0\\1\\2},
\bar{\msmat{4\\3\\1}} = \msmat{1\\2\\1},
\bar{\msmat{8\\4\\2}} = \msmat{4\\2\\2}.
\]
Therefore, the crank function $c=\eta\circ\psi\circ\phi$ that witnesses $p(6mk'+ 2m-2,3) \equiv 0 \pmod m$ is given by the formula
\begin{eqnarray*}
c(\lambda) &=& \tau_2 + (\bar{\mu}_3-1)(\tau_1+1) + \bar{\mu}_2(2+k) \;\;\; \mmod m ,
\end{eqnarray*}
where $\mu$ and $\tau$ are the box remainder and box quotient as introduced in Definition~\ref{def:decomposition}.
From the division with remainder $\lambda = V_3\tau + \mu$, it is easy to derive explicit formulas for $\mu$ and $\tau$,
expressed in terms of the floor and fractional part functions  $\floor{x}$ and $\{x\}$.
In the case $r'=2m-2$, this leads to the following formula for the crank function:
\begin{align*}
c(\lambda) &=& \floor{\frac{\bar{\lambda}_2}{3}}
    + \left( 2\fract{\frac{\bar{\lambda}_3}{2}} + 2\left[\frac{\bar{\lambda}_3}{2}\in\Z\right] - 1 \right)
      \left(\floor{ \frac{ \bar{\lambda}_1 }{ 6 } } + 1 \right)
    + 3\fract{\frac{\bar{\lambda}_2}{3}}(2+k)
    \; \mmod m ,
\end{align*}
where $k=\tau_1+\tau_2+\tau_3 = \floor{\frac{\bar{\lambda}_1}{6}} + \floor{\frac{\bar{\lambda}_2}{3}} + \floor{\frac{\bar{\lambda}_3}{2}}-\left[\frac{\bar{\lambda}_3}{2}\in\Z\right]$ and
$[F]$ denotes the value 1 if the statement $F$ is true and 0 otherwise.

\section{Largest Part Minus Smallest Part is a Supercrank}
\label{labeled origami}

\begin{figure*}[t]

  \begin{center}
  \includegraphics[angle=270,width=14cm]{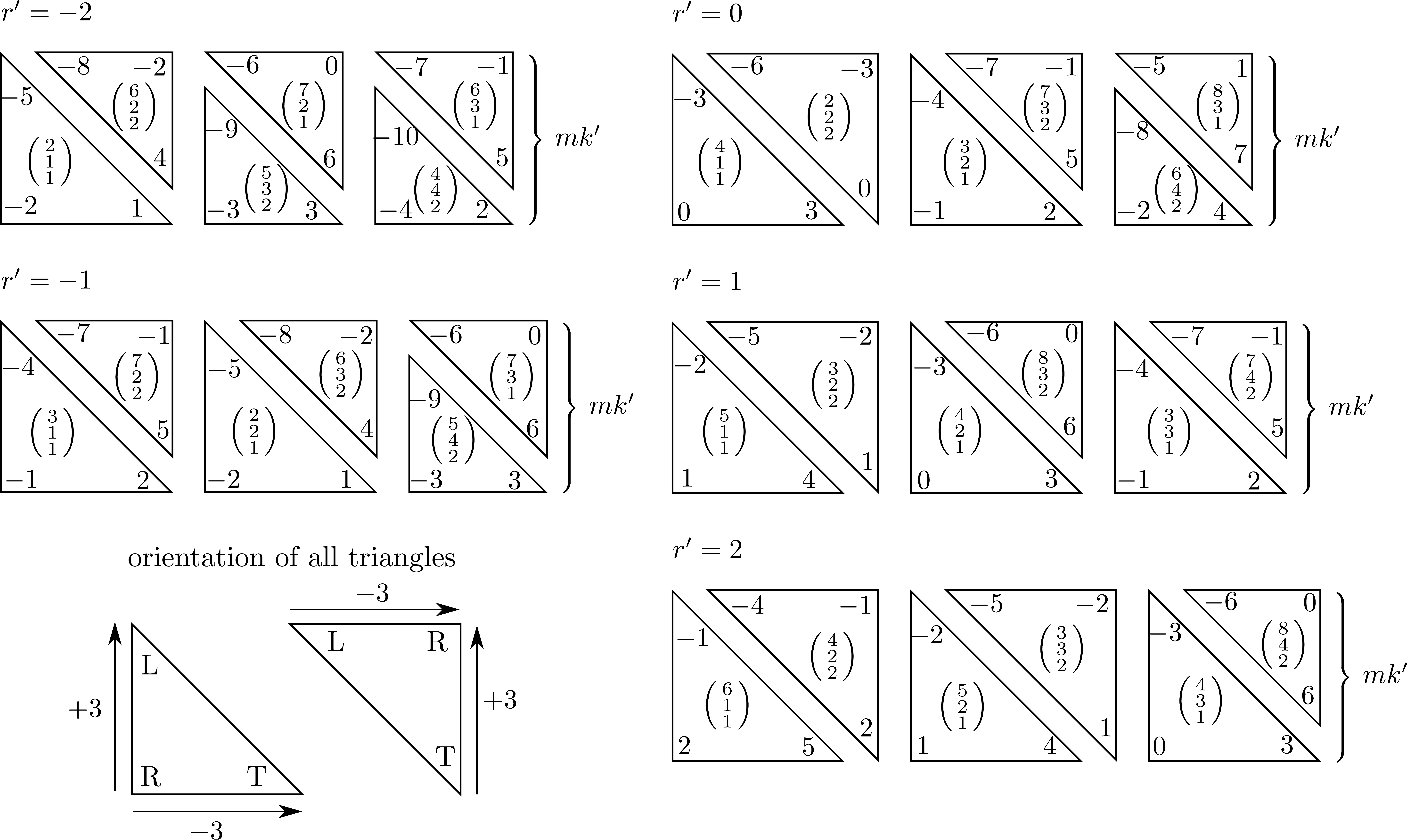}
  \end{center}
  \caption[]{ \label{fig:labeled-origami-cases-vertical} The cases $r'\in\{0,\pm 1,\pm 2\}$. All triangles are oriented as shown in the lower-left of the figure.
   The vertices of each triangle are labeled with their respective crank value modulo $m$. The triangles are labeled by the corresponding $\mu$ vector. In every column of every rectangle the crank value increases by $3$ modulo $m$ at every step, wrapping around at the top, which provides cycles witnessing divisibility. Along each row, the crank value typically decreases by $3$, but there are two jumps in each row; this does not affect the cycles, however.
  }
\end{figure*}

\begin{figure*}[t]

  \begin{center}
  \includegraphics[angle=270,width=14cm]{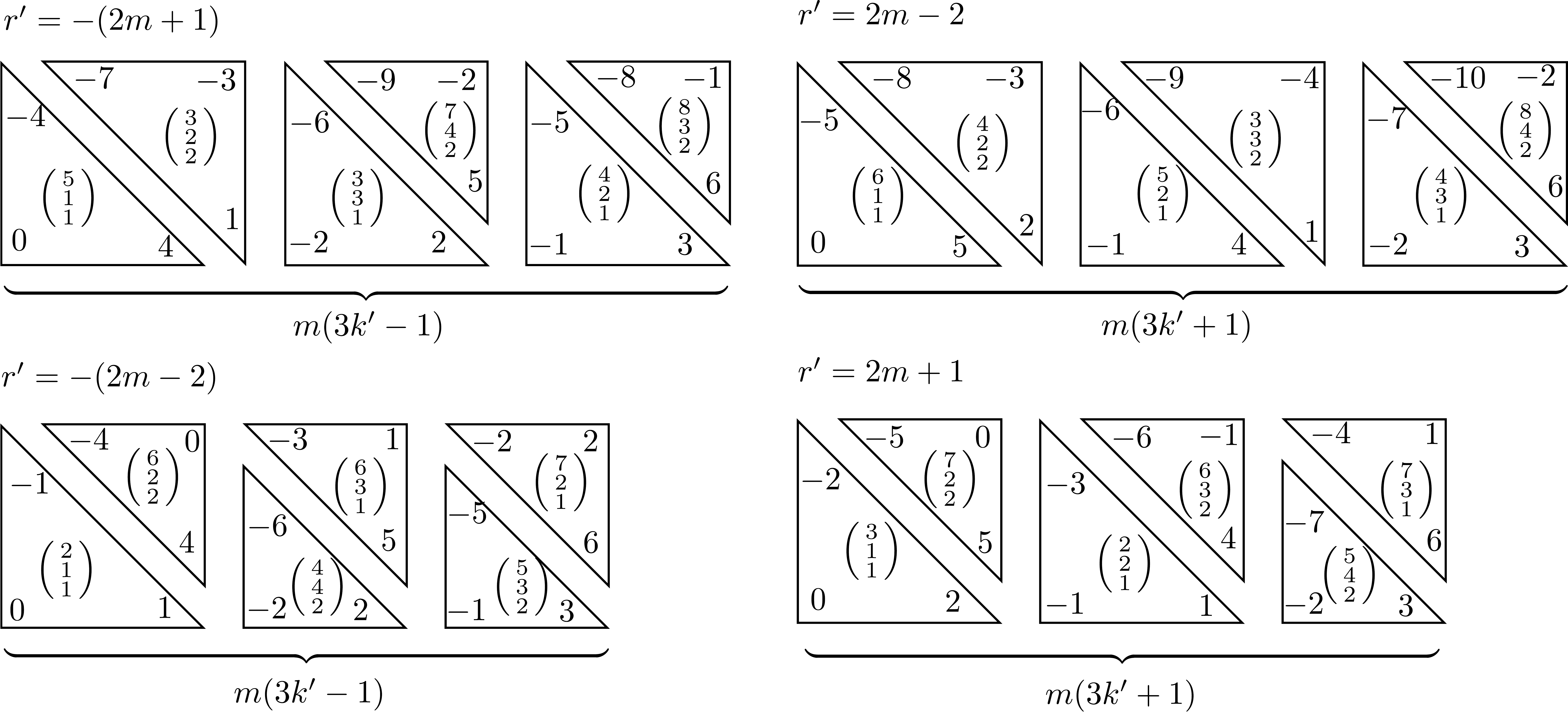}
  \end{center}

  \caption[]{ \label{fig:labeled-origami-cases-horizontal} The cases $r'\in\{\pm(2m-2),\pm (2m+1)\}$. All triangles are oriented as shown in Figure~\ref{fig:labeled-origami-cases-vertical}. The vertices of each triangle are labeled with their respective crank value modulo $m$. The triangles are labeled by their corresponding box remainder, $\mu$. In every row of every rectangle the crank value decreases by $3$ modulo $m$ at every step, wrapping around at the right, which provides cycles witnessing divisibility. Along each column, the crank value increases by $3$ modulo $m$ without wrapping around at the top; this is not relevant for the cycles, however.
  }
\end{figure*}

In this section, we use the methods developed in Sections~\ref{sec:ehrhart} and \ref{sec:Ehrhart-cranks} to prove Theorem~\ref{main theorem}, which we restate here.

\begin{reptheorem}{main theorem}
Let $m=6j-1$ be prime. Then $c_{LS}$, largest part minus smallest part modulo $m$, is a supercrank witnessing $m|p(n,3)$ for each and every $n\in  \Z_{\geq 0}$ for which this divisibility holds, as characterized in Proposition~\ref{6j-1prop}.
\end{reptheorem}

\begin{proof}
For all $n$ characterized in Proposition~\ref{6j-1prop}, we have to show that
$c_{LS}(\lambda)=\lambda_1 -\lambda_3 \mmod m$, witnesses $m|p(n,3)$; i.e.,
\begin{eqnarray}
  \label{eqn:crank-property}
  \#c_{LS}^{-1}(0) = \#c_{LS}^{-1}(1) = \ldots = \#c_{LS}^{-1}(m-1).
\end{eqnarray}
To witness (\ref{eqn:crank-property}) combinatorially, it suffices to find bijections $f_i:c_{LS}^{-1}(i)\rar c_{LS}^{-1}(i+1 \mmod m)$ for all $i=0,\ldots,m-1$.
One way to construct these bijections is to find a single involution $f:P(n,3)\rar P(n,3)$ with the property that $c_{LS}(f(\lambda))=c_{LS}(\lambda)+\delta \mmod m$, where $\delta$ is a constant relatively prime to $m$.
Note that this property implies that on each cycle of the permutation $f$, all crank values appear equally often. For brevity, we say that $f$ \emph{cycles}
the crank values.

The construction of $f$ is summarized in Figures~\ref{fig:labeled-origami-cases-vertical} and
\ref{fig:labeled-origami-cases-horizontal}: We first decompose the relevant sets $P(n,3)$ into triangles
$T_k$ as in Section~\ref{sec:ehrhart}.
Then we reassemble these triangles into rectangles $R(k')$ as in Section~\ref{sec:Ehrhart-cranks}.
This time, however, we take additional care to make sure that values of $c_{LS}$ change by a constant $\delta$ at each step along the columns (cases $r'\in\{0,\pm1,\pm2\}$,
shown in Figure~\ref{fig:labeled-origami-cases-vertical}) or rows (cases $r'\in\{\pm(2m-2),\pm(2m+1)\}$, shown in Figure~\ref{fig:labeled-origami-cases-horizontal}) of the rectangles, wrapping around at the boundary.
The permutation $f$ that cycles the crank values is thus given by taking one step to the right or taking one step up in the
rectangles
given in Figures~\ref{fig:labeled-origami-cases-vertical} and \ref{fig:labeled-origami-cases-horizontal}.
To make these ideas precise, we now go into detail for the case $r'=2m-2$. The other cases are analogous.

Let $n=6mk' + r'$ for $r'=2m-2$. In this case, $k=mk'+\frac{m-2}{3}$ and $r=2$. Given the decomposition $\phi_{k,r}:P(n,3)\rar Q_{k,r}$ defined in Section~\ref{sec:ehrhart} and the triangle arrangement $\psi$ given in Figure~\ref{fig:labeled-origami-cases-horizontal} for this case, we define $f$ by
\begin{eqnarray}
\label{eqn:f}
  f(\lambda) = (\psi\circ\phi_{k,r})^{-1}(\psi\circ\phi_{k,r}(\lambda) + \msmat{1\\0});
\end{eqnarray}
i.e., we get from $\lambda$ to the next partition in the cycle by first finding the position of $\lambda$ in the rectangle $R(k')$, taking one step right in this rectangle, and then making note of the corresponding partition. Moreover, when we reach the right edge of the rectangle, we wrap around and return to left edge of the rectangle in the same line. More precisely, the addition $+ \msmat{1\\0}$ in (\ref{eqn:f}) is to be understood modulo $m(3k'-1)$ in the first coordinate.

\begin{figure*}[t]
  \begin{center}
  \includegraphics[angle=270,width=3cm]{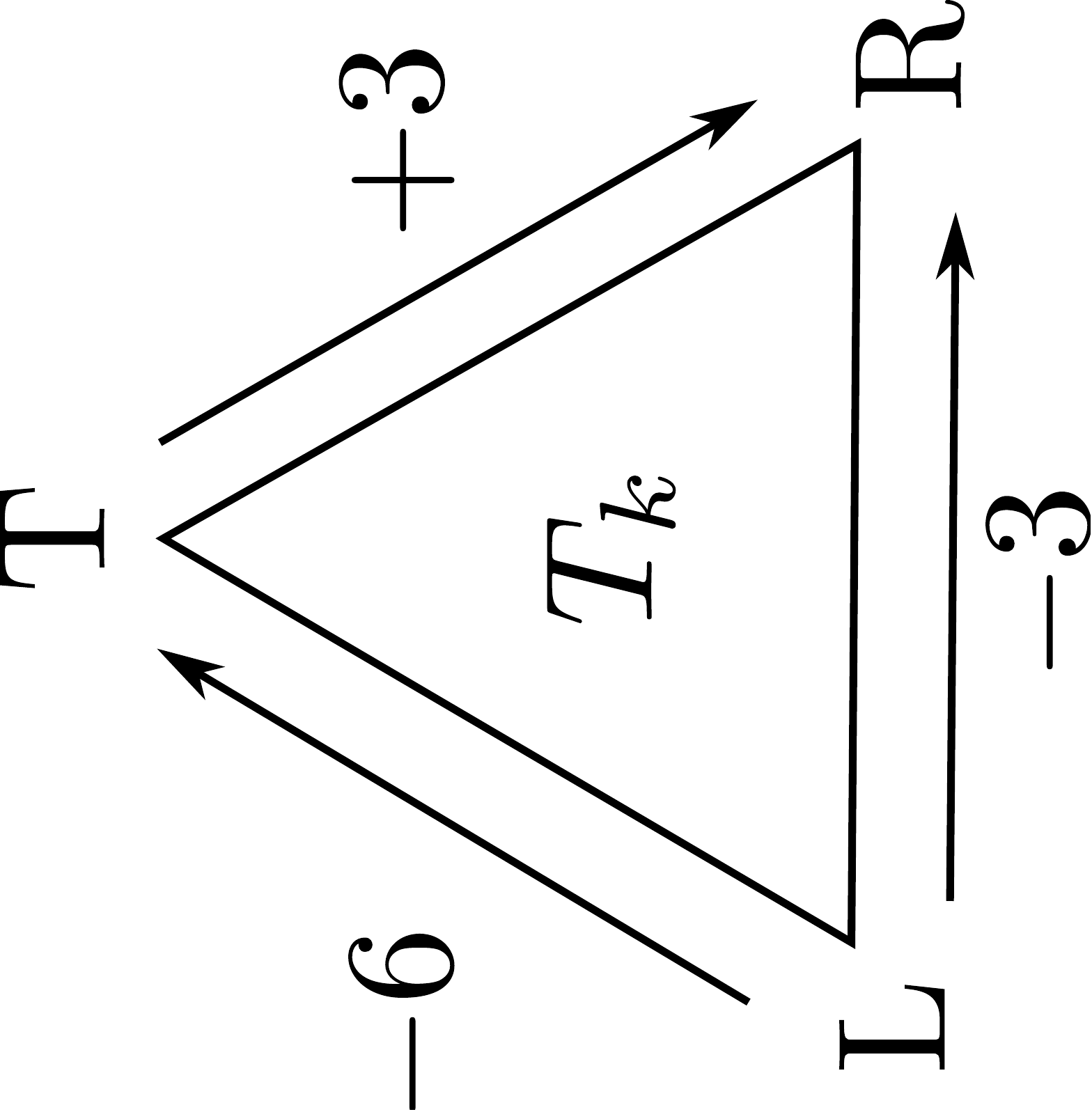}
  \end{center}

  \caption[]{ \label{fig:steps-in-triangle} Let L (left), R (right), T (top) denote the vertices $\msmat{k\\0\\0}$, $\msmat{0\\k\\0}$, $\msmat{0\\0\\k}$ of $T_k$, respectively. Taking a single step from L to T changes the value of the crank $c(\lambda)=\lambda_1 - \lambda_3$ by $-6$ modulo $m$, i.e., $c\left(V\left(\tau+\msmat{-1\\0\\1}\right) + \mu\right) - c\left(V\tau+\mu\right) = c\left(V\msmat{-1\\0\\1}\right) = -6 \mmod m$, and similarly for the other steps shown in the figure.
  }
\end{figure*}

The function $f$ is a permutation by construction. Moreover, $c_{LS}(f(\lambda))=c_{LS}(\lambda)-3 \mmod m$, which we can see as follows.
First, we observe that taking one step in any triangle $T_k$ changes the crank value as shown in Figure~\ref{fig:steps-in-triangle}.
Thus, if triangles are oriented as shown in the bottom left of Figure~\ref{fig:labeled-origami-cases-vertical}, then within each triangle the crank value will consistently change by $-3$ when taking one step right. Therefore, all we have to ensure is that the triangle arrangement $\psi$ is such that the crank value also changes by $-3$ when stepping right from one triangle to the next (and at the boundary). To this end, it is important to observe that for a fixed $\mu$ the crank values at the vertices $\mathrm{L}=\msmat{k\\0\\0}$, $\mathrm{R}=\msmat{0\\k\\0}$ and $\mathrm{T}=\msmat{0\\0\\k}$ of $T_k$ are fixed constants modulo $m$ that do not depend on $k$ or $m$. For example, we can compute that for $r=2m-2$ and $\mu=\msmat{6\\1\\1}$ we have
\begin{eqnarray*}
  c_{LS}(\phi_{k,r}^{-1}(\mu,\mathrm{L})) &=& -5 \mod m,\\
  c_{LS}(\phi_{k,r}^{-1}(\mu,\mathrm{R})) &=& 0 \mod m,\\
  c_{LS}(\phi_{k,r}^{-1}(\mu,\mathrm{T})) &=& 5 \mod m.
\end{eqnarray*}
This data is recorded for all vertices of all triangles in Figures~\ref{fig:labeled-origami-cases-vertical} and \ref{fig:labeled-origami-cases-horizontal}.
In the case $r=2m-2$, we can see that moving horizontally from the triangle for $\mu=\msmat{6\\1\\1}$ to the triangle for $\mu=\msmat{4\\2\\2}$ the crank value
changes by $-3$ as desired, and similarly for all other steps to the right between different triangles. In particular, going from the rightmost edge of the rectangle to the leftmost edge, the crank value also changes by $-3$. This completes the proof that $f$ cycles the crank values and thus shows that $c_{LS}$ is a supercrank for $p(n,3)$.
\end{proof}

Above, we have used the methods developed in Sections~\ref{sec:ehrhart} and \ref{sec:Ehrhart-cranks} to prove that $c_{LS}$ is a supercrank.
Thus the question arises whether the crank $c_{LS}$ is actually an Ehrhart crank in the sense of Theorem~\ref{thm:unlabeled-technical};
i.e., whether $c_{LS}(\lambda)=c(\lambda)=\eta\circ\psi\circ\phi(\lambda)$ for the triangle arrangements $\psi$ given in the figures.
For the cases $r'\in\{\pm(2m-2),\pm(2m+1)\}$, this is true immediately when choosing $\eta(x,y)=-3x+3y+\gamma$, where $\gamma$ is the value in the lower-left corner of the respective rectangle.
However, for the cases $r'\in\{0,\pm1,\pm2\}$, this does not work as the crank value $c_{LS}$ does not increase linearly when moving to the right: in each case there are two ``jumps'' when moving from one small sub-rectangle to the next. While this does not affect the above proof at all, it means that no affine linear function $\eta$ can produce $c_{LS}$ as an Ehrhart crank. However, if we generalize the notion of an Ehrhart crank to allow for functions $\eta$ that are piecewise linear, then the crank $c_{LS}$ is
an Ehrhart crank.

\section{A Second Proof of Theorem \ref{main theorem}.}
\label{sec:mountain}

We provide an alternate proof of Theorem \ref{main theorem} by cycling the partitions of $P(n,3)$ utilizing vector representations as in Figure \ref{fig:partition-triangle}.
The result is bijection from $P(n,3)$ back to itself.

\begin{figure}
\begin{center}
\includegraphics[angle=270,width=13cm]{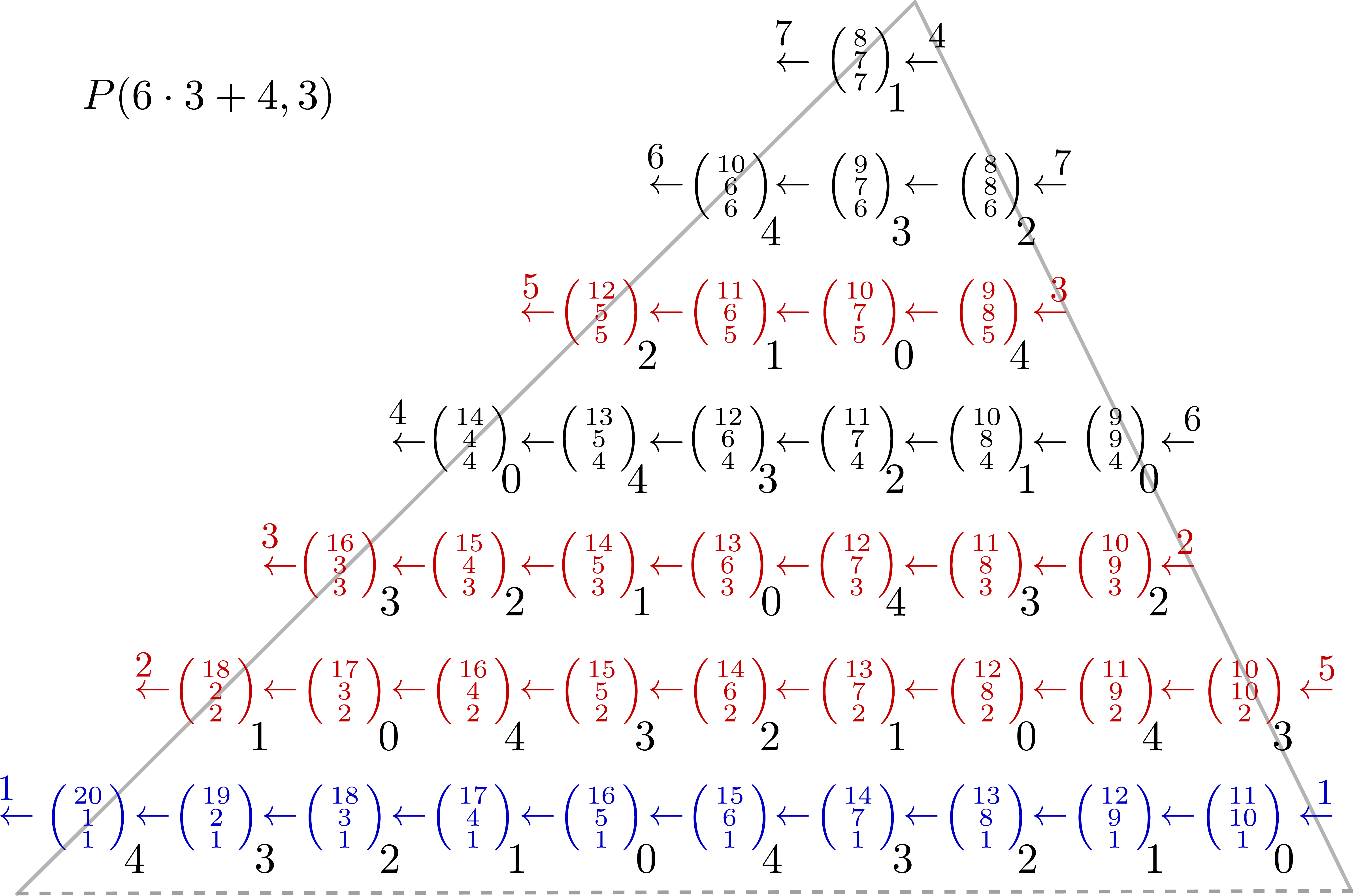}
\end{center}
 \caption[]{ \label{fig:mountain} There are three distinct $m$-cycles for the case $p(6\cdot3+4,3)=40$ with $m=5,~k=1$ and $r'=-(2m-2))$.
They are formed by 1) cycling the 20 partitions in the horizontal rows 2, 3, and 5 of $P(22,3)$ (shown in red), 2) cycling the 10 partitions in the horizontal rows 4, 6, and 7 of $P(22,3)$ (shown in black), and 3) cycling the 10 partitions in the horizontal row 1 of $P(22,3)$ (shown in blue). The permutation on the rows defined by the bijection between the left and right boundaries is shown by the labels on the arrows at the sides, see bullet $r'=-(2m-2)$ in this section. The crank values of the partitions are shown as subscripts on the partitions, e.g., the partition $22=17+3+2$ has crank value $0$.}
\end{figure}

\begin{proof}[Alternative Proof of Theorem~\ref{main theorem}]
Let $n$ be such that $m | p(n,3)$.
By \rprop{6j-1prop}, we may write $n = 6mk + r'$ for some $r' \in \{0,\pm 1,\pm 2,\pm (2m-2), \pm (2m+1)\}$.
Following the approach taken in Section~\ref{labeled origami}, we show that $c_{LS}$ is a crank by constructing a permutation $f$ on $P(n,3)$ with the property $c_{LS}(f(\lambda))=c_{LS}(\lambda)+1\mmod m$. A partition $\lambda\in P(n,3)$
is on the left border $B$ of the partition triangle if and only if $\lambda_2=\lambda_3$. As illustrated in Figure~\ref{fig:mountain}, we now define $f$ as follows:
\begin{enumerate}
\item For $\lambda_2\not=\lambda_3$,~ $f(\lambda)=f\left(\msmat{\lambda_1\\ \lambda_2 \\ \lambda_3}\right)=\msmat{\lambda_1\\ \lambda_2-1 \\ \lambda_3+1}$; i.e., if $\lambda$ is not on the left border of $P(n,3)$ the map $f$ translates $\lambda$ one unit to the left within $P(n,3)$.
\item For $\lambda_2=\lambda_3$,~ $f(\lambda)$ is defined by case distinction on $r'$ (see list below).
If $\lambda$ is on the left border of $P(n,3)$, then $f$ maps $\lambda$ to the right border of $P(n,3)$ such that the value of $c_{LS}$ increases by 1. \label{ii}
\end{enumerate}
It is now straightforward to check that $f$ is injective and that $c_{LS}(f(\lambda))=c_{LS}(\lambda)+1\mmod m$.
For partition $\lambda\not\in B$ this is immediate, and for $\lambda\in B$ this follows from the list of case distinctions below.
Thus $f$ cycles the values of $c_{LS}$ as in Section~\ref{labeled origami}, which proves that $c_{LS}$ is a supercrank for
$p(n,3)$.
\end{proof}

Figure \ref{fig:mountain} provides an example of the three distinct $m-$cycles for the case $p(22,3)=40$ with $m=5,~k=1$, and $r'=-(2m-2)$.

We note that a partition $\lambda$ on the left border $B$ of $P(n,3)$ has the form $\lambda=\msmat{6mk+r'-2\lambda_3\\ \lambda_3 \\ \lambda_3}$.
For such partitions, we define the map $f$ by the following list of nine case distinctions depending on $r'$:

\begin{itemize}
\item $r'=-(2m+1)$

\begin{equation}\nonumber
    f(\lambda)=\left\{
      \begin{array}{cc}
      \msmat{\frac{6mk+r'+1}{2} - \lambda_3 \\ \frac{6mk+r'+1}{2} - (\lambda_3+1) \\ 2\lambda_3} & 1\leq \lambda_3 \leq mk+\floor{\frac{r'}{6}}
\\ ~\\
      \msmat{\frac{6mk+r'+1}{2} - \left(\lambda_3-mk-\floor{\frac{r'}{6}}\right) \\ \frac{6mk+r'+1}{2} - \left(\lambda_3-mk-\floor{\frac{r'}{6}}\right) \\ 2\left(\lambda_3-mk-\floor{\frac{r'}{6}}\right)-1}  & mk+\floor{\frac{r'}{6}}+1\leq \lambda_3 \leq 2\left(mk+\floor{\frac{r'}{6}}\right)
      \end{array}
 \right.
 \end{equation}

\item $r'=-(2m-2)$
\begin{equation}\nonumber
    f(\lambda)=\left\{
      \begin{array}{cc}
      \msmat{\frac{6mk+r'}{2} - (\lambda_3-1) \\ \frac{6mk+r'}{2} - \lambda_3 \\ 2\lambda_3-1} & 1\leq \lambda_3 \leq mk+\floor{\frac{r'}{6}}+1
\\ ~\\
      \msmat{\frac{6mk+r'}{2} - \left(\lambda_3-mk-\floor{\frac{r'}{6}}-1\right) \\ \frac{6mk+r'}{2} - \left(\lambda_3-mk-\floor{\frac{r'}{6}}-1\right) \\ 2\left(\lambda_3-mk-\floor{\frac{r'}{6}}-1\right)} &  mk+\floor{\frac{r'}{6}}+2\leq \lambda_3 \leq 2(mk+\floor{\frac{r'}{6}})+1
      \end{array}
 \right.
 \end{equation}

\item $r'=-2$
\begin{equation}\nonumber
    f(\lambda)=\left\{
      \begin{array}{cc}
     \msmat{\frac{6mk}{2} - (\lambda_3+1) \\ \frac{6mk}{2} - (\lambda_3+1) \\ 2\lambda_3} &
1\leq \lambda_3 \leq mk-1
\\ ~\\
     \msmat{\frac{6mk}{2} - \left(\lambda_3+1-2j\right) \\ \frac{6mk}{2} - \left(\lambda_3+2-2j\right) \\ 2\left(\lambda_3-2j\right)+1} & mk\leq \lambda_3 \leq mk-1+2j
\\ ~\\
     \msmat{\frac{6mk}{2} - \left(\lambda_3+1-2j-mk\right) \\ \frac{6mk}{2} - \left(\lambda_3+2-2j-mk\right) \\ 2\left(\lambda_3+1-2j-mk\right)-1} & mk+2j\leq \lambda_3 \leq 2mk-1
      \end{array}
 \right.
 \end{equation}

\item $r'=-1$
\begin{equation}\nonumber
    f(\lambda)=\left\{
      \begin{array}{cc}
    \msmat{\frac{6mk}{2} - \lambda_3 \\ \frac{6mk}{2} - (\lambda_3+1) \\ 2\lambda_3} &
1\leq \lambda_3 \leq mk-1
\\~\\
 \msmat{\frac{6mk}{2} - \left(\lambda_3+1-4j\right) \\ \frac{6mk}{2} - \left(\lambda_3+1-4j\right) \\ 2\left(\lambda_3+1-4j\right)-1} &
 mk\leq \lambda_3 \leq mk+4j-1
\\~\\
 \msmat{\frac{6mk}{2} - \left(\lambda_3+1-4j-mk\right) \\ \frac{6mk}{2} - \left(\lambda_3+1-4j-mk\right) \\ 2\left(\lambda_3+1-4j-mk\right)-1} &
mk+4j\leq \lambda_3 \leq 2mk-1
      \end{array}
 \right.
 \end{equation}

\item $r'=0$
\begin{equation}\nonumber
    f(\lambda)=\left\{
      \begin{array}{cc}
   \msmat{\frac{6mk}{2} - (\lambda_3-1+2j) \\ \frac{6mk}{2} - \left(\lambda_3+2j\right) \\ 2\left(\lambda_3+2j\right)-1} &
1\leq \lambda_3 \leq mk-2j
\\ ~\\
\msmat{\frac{6mk}{2} - \left(\lambda_3-1+2j-mk\right) \\ \frac{6mk}{2} - \left(\lambda_3+2j-mk\right) \\ 2\left(\lambda_3+2j-mk\right)-1} &
mk-2j+1\leq \lambda_3 \leq mk
\\~\\
\msmat{\frac{6mk}{2} - \left(\lambda_3-2j\right) \\ \frac{6mk}{2} - \left(\lambda_3-2j\right) \\ 2\left(\lambda_3-2j\right)} &
mk+1\leq \lambda_3 \leq mk+2j
\\~\\
\msmat{\frac{6mk}{2} - \left(\lambda_3-2j-mk\right) \\ \frac{6mk}{2} - \left(\lambda_3-2j-mk\right) \\ 2\left(\lambda_3-2j-mk\right)} &
mk+2j+1\leq \lambda_3 \leq 2mk
      \end{array}
 \right.
 \end{equation}

\item $r'=1$
\begin{equation}\nonumber
    f(\lambda)=\left\{
      \begin{array}{cc}
  \msmat{\frac{6mk}{2} - (\lambda_3-1) \\ \frac{6mk}{2} - (\lambda_3-1) \\ 2\lambda_3-1} &
1\leq \lambda_3 \leq mk
\\~\\
\msmat{\frac{6mk}{2} - \left(\lambda_3-1-2j\right) \\ \frac{6mk}{2} - \left(\lambda_3-2j\right) \\ 2\left(\lambda_3-2j\right)} &
mk+1\leq \lambda_3 \leq mk+2j
\\~\\
\msmat{\frac{6mk}{2} - \left(\lambda_3-1-2j-mk\right) \\ \frac{6mk}{2} - \left(\lambda_3-2j-mk\right) \\ 2\left(\lambda_3-2j-mk\right)} &
mk+2j+1\leq \lambda_3 \leq 2mk
      \end{array}
 \right.
 \end{equation}

\item $r'=2$
\begin{equation}\nonumber
    f(\lambda)=\left\{
      \begin{array}{cc}
\msmat{\frac{6mk}{2} - (\lambda_3 - 2) \\ \frac{6mk}{2} - (\lambda_3-1) \\ 2\lambda_3-1} &
1\leq \lambda_3 \leq mk
\\~\\
 \msmat{\frac{6mk}{2} - \left(\lambda_3-1-4j\right) \\ \frac{6mk}{2} - \left(\lambda_3-1-4j\right) \\ 2\left(\lambda_3-4j\right)} &
mk+1\leq \lambda_3 \leq mk+4j
\\~\\
 \msmat{\frac{6mk}{2} - \left(\lambda_3-1-4j-mk\right) \\ \frac{6mk}{2} - \left(\lambda_3-1-4j-mk\right)\\ 2\left(\lambda_3-4j-mk\right)} &
mk+4j+1\leq \lambda_3 \leq 2mk
      \end{array}
 \right.
 \end{equation}

\item $r'=2m-2$
\begin{equation}\nonumber
    f(\lambda)=\left\{
      \begin{array}{cc}
 \msmat{\frac{6mk+r'}{2} - \lambda_3 \\ \frac{6mk+r'}{2} - \lambda_3 \\ 2\lambda_3} &
1\leq \lambda_3 \leq mk+\floor{\frac{r'}{6}}
\\~\\
 \msmat{\frac{6mk+r'}{2} - \left(\lambda_3-1-mk-\floor{\frac{r'}{6}}\right) \\ \frac{6mk+r'}{2} - \left(\lambda_3-mk-\floor{\frac{r'}{6}}\right) \\ 2\left(\lambda_3-mk-\floor{\frac{r'}{6}}\right)-1} &
mk+\floor{\frac{r'}{6}}+1\leq \lambda_3 \leq 2(mk+\floor{\frac{r'}{6}})
      \end{array}
 \right.
 \end{equation}

\item $r'=2m+1$
\begin{equation}\nonumber
    f(\lambda)=\left\{
      \begin{array}{cc}
 \msmat{\frac{6mk+r'+1}{2} - \lambda_3  \\ \frac{6mk+r'+1}{2} - \lambda_3 \\ 2\lambda_3-1} &
1\leq \lambda_3 \leq mk+\floor{\frac{r'}{6}}+1
\\~\\
  \msmat{\frac{6mk+r'+1}{2} - \left(\lambda_3-1-mk-\floor{\frac{r'}{6}}\right) \\ \frac{6mk+r'+1}{2} - \left(\lambda_3-mk-\floor{\frac{r'}{6}}\right) \\ 2\left(\lambda_3-1-mk-\floor{\frac{r'}{6}}\right)} &
mk+\floor{\frac{r'}{6}}+2\leq \lambda_3 \leq 2(mk+\floor{\frac{r'}{6}})+1.      \end{array}
 \right.
 \end{equation}
\end{itemize}

\section{Applications to other primes}
\label{section6j+1}

In the previous sections, we restricted our attention to primes $m \equiv -1 \pmod 6$
because the largest part minus the smallest part modulo $m$ is a supercrank
witnessing every instance of divisibility of $p(n,3)$ by any such $m$.
For every prime $m' \equiv 1 \pmod 6$, all of the same techniques still apply
for proving and witnessing some divisibility of $p(n,3)$ modulo $m'$,
but in this case we no longer know of a supercrank for even one such $m'$.
As it turns out, the largest part minus the smallest part modulo $m'$
still witnesses \emph{most} of the divisibility of $p(n,3)$ by $m'$;
however, $p(n,3)$ enjoys two more arithmetic progressions of divisibility modulo primes $\equiv 1 \pmod 6$,
and none of the techniques from Sections \ref{sec:Ehrhart-cranks}, \ref{labeled origami}, or \ref{sec:mountain} can account for this extra divisibility.
We describe this extra divisibility with an analog of \rprop{6j-1prop} below.

\begin{proposition}\label{6j+1prop}
Let $m'  = 6j+1$ be prime.
Then
$
p(n,3) \equiv 0 \pmod {m'}$   if and only if
$$
n \equiv \pm (0,1,2,2m'-1,2m'+2,3m'+s(m'-1)) \pmod {6m'},
$$
where $s$ satisfies $s^2 \equiv -3 \pmod {m'}$.

\end{proposition}

Our combinatorial methods for proving Theorem \ref{main theorem}
work here to prove

\begin{theorem}\label{6j+1 theorem}
Let $m'=6j+1$ be prime.
Then $c_{LS}$, largest part minus smallest part modulo $m'$,
is a crank witnessing
$p(6m'k\pm 0,1,2,2m'-2,2m'+1;3)\equiv 0\pmod{m'}$.
\end{theorem}
However, the congruences $p(6m'k\pm 3m'+s(m'-1), 3)\equiv 0\pmod{m'}$ from Proposition \ref{6j+1prop} are resistant to our methods.
We have no combinatorial proof of $p(6m'k\pm 3m'+s(m'-1), 3)\equiv 0\pmod{m'}$, nor do we know of a crank statistic that provides a combinatorial witness.
One explanation for why the same techniques do not provide a combinatorial proof of divisibility
is that on the Ehrhart side, since $3m + s(m'-1) \equiv 3 \pmod{6}$, our decomposition of $P(n,3)$
now produces triangles of three different sizes instead of just two.
Because of these three different sizes, it is
impossible to rearrange the six triangles into a rectangle,
and so no Ehrhart crank is forthcoming.

Regarding divisibility of $p(n,3)$ by the remaining two primes, 2 and 3,
supercranks in both cases were found in \cite{EichhornKronholm-forthcoming}.
It is interesting to observe that the parity of the largest part minus smallest part
again appears as a supercrank witnessing every instance of $2 | p(n,3)$.
However, the supercrank witnessing each and every instance of $3 | p(n,3)$ is a different statistic; it is simply the largest part modulo 3.

\section{Conclusion}

Dyson's 70-year-old request for a \emph{bijective} proof that the rank witnesses Ramanujan's congruences modulo 5 and 7
remains unfulfilled.
Since Garvan, Kim, and Stanton made the remarkable stride
of discovering purely bijective proofs of Ramanujan's first four congruences
along with new crank statistics witnessing them twenty-five years ago \cite{GKS},
surprisingly little has been done treating congruences for partition functions using bijective methods.

In this paper, we have taken two approaches
to provide bijective proofs and combinatorial witnesses for
every instance of divisibility of $p(n,3)$ by primes $m \equiv -1 \pmod 6$
as well as $9/11$ths of the divisibility by primes $m \equiv 1 \pmod 6$.
The first approach is
rooted in
Ehrhart theory.
Remarkably, although there are many
Ehrhart cranks one can construct
using this approach
that witness each of these divisibilities,
one very well-poised crank, $c_{LS}$, the largest part minus smallest part taken modulo $m$,
witnesses every single one.
The second approach relies upon creating a permutation
of partitions whose cycle decomposition includes only cycles
with lengths divisible by $m$,
a fact also witnessed by $c_{LS}$.
It remains to be seen if the fact that we have a supercrank for primes $\equiv-1 \pmod 6$
while the same crank only witnesses $9/11$ths of the divisibility by
primes $\equiv1 \pmod 6$ will lead to a critical insight,
or if the discrepancy is just an unremarkable property of the integers.

Our hope here is that by revealing the structure of these sets of partitions in a new way,
and in a way that reveals \emph{why} certain crank statistics witness divisibility
of $p(n,3)$, we will open the gateway to a deeper understanding of
other partition functions and their divisibility.
We have restricted our attention here to studying the geometry and combinatorics
governing $p(n,3)$.
However, these methods
can still be applied when $d>3$,
and in \cite{BreuerEichhornKronholm-forthcoming} the authors treat this more general
setting.

The authors would like to thank Peter Paule and the Research Institute for Symbolic Computation
for their generous support of this research, including support of the first and third
authors as postdoctoral fellows, and support for the second author on two research
visits.
RISC was an especially appropriate facility for this research,
both because of RISC's excellent history of combinatorial research,
and because the remarkable properties of $c_{LS}$ were originally discovered while
searching for cranks for $p(n,d)$ computationally.

\setlength{\bibsep}{0pt}
\renewcommand{\bibfont}{\small}
\bibliographystyle{abbrv}
\bibliography{TESTreferences}

\end{document}